\documentclass[reqno]{amsart}

\pdfoutput=1

\usepackage[francais, english]{babel}
\usepackage[alphabetic]{amsrefs}
\usepackage{amsmath}
\usepackage{amsfonts}
\usepackage{amssymb}
\usepackage{enumerate}
\usepackage{amstext}
\usepackage{amsbsy}
\usepackage{amsopn}
\usepackage{bbm,amsthm}
\usepackage{amscd}
\usepackage[pdftex]{color}
\usepackage{amsxtra}
\usepackage{upref}
\usepackage{epstopdf}
\usepackage{graphicx,color}
\usepackage{hyperref}
\usepackage{longtable}
\usepackage{graphicx}

\usepackage{amsmath, amsfonts, amssymb,amscd, amstext, amsthm, 
mathtools, bezier, mathrsfs,enumerate}
\usepackage{url}
\usepackage{float}
\usepackage{caption}
\usepackage{tikz}
\usetikzlibrary{intersections,positioning,patterns,arrows,decorations.markings}
\usepackage[fit]{truncate}
\usepackage{anyfontsize}
\usepackage{t1enc}
\usepackage{synttree}
\usepackage{pst-plot}
\psset{algebraic}

\DeclareFontFamily{OML}{rsfs}{\skewchar\font'177}
\DeclareFontShape{OML}{rsfs}{m}{n}{ <5> <6> rsfs5 <7> <8> <9> rsfs7
  <10> <10.95> <12> <14.4> <17.28> <20.74> <24.88> rsfs10 }{}
\DeclareMathAlphabet{\mathfs}{OML}{rsfs}{m}{n}

\newtheorem{theorem}{Theorem}
\newtheorem{lemma}[theorem]{Lemma}
\newtheorem{proposition}[theorem]{Proposition}
\newtheorem{corollary}[theorem]{Corollary}

\theoremstyle{definition}
\newtheorem{definition}[theorem]{Definition}

\newtheorem*{teoBBB}{Theorem B} 
\newtheorem*{theoremm}{Theorem}
\newtheorem{example}{Example}
\newtheorem{theoremletra}{\textbf{Theorem}}
\renewcommand{\thetheoremletra}{\Alph{theoremletra}}

\theoremstyle{remark}

\numberwithin{equation}{section}
\numberwithin{theorem}{section}

\newcommand{\intav}[1]{\mathchoice {\mathop{\vrule width 6pt height 3 pt depth  -2.5pt
\kern -8pt \intop}\nolimits_{\kern -6pt#1}} {\mathop{\vrule width
5pt height 3  pt depth -2.6pt \kern -6pt \intop}\nolimits_{#1}}
{\mathop{\vrule width 5pt height 3 pt depth -2.6pt \kern -6pt
\intop}\nolimits_{#1}} {\mathop{\vrule width 5pt height 3 pt depth
-2.6pt \kern -6pt \intop}\nolimits_{#1}}}

\newcommand{\intavl}[1]{\mathchoice {\mathop{\vrule width 6pt height 3 pt depth  -2.5pt
\kern -8pt \intop}\limits_{\kern -6pt#1}} {\mathop{\vrule width 5pt
height 3  pt depth -2.6pt \kern -6pt \intop}\nolimits_{#1}}
{\mathop{\vrule width 5pt height 3 pt depth -2.6pt \kern -6pt
\intop}\nolimits_{#1}} {\mathop{\vrule width 5pt height 3 pt depth
-2.6pt \kern -6pt \intop}\nolimits_{#1}}}

\newcommand{\cc}{\mathscr{C}}

\newcommand{\N}{\mathbb{N}}

\newcommand{\li}{\mathscr{L}}
\newcommand{\pp}{\mathscr{P}}

\begin{document}


\title[Kneading sequences for toy models of H\'enon maps]{Kneading sequences for toy \\ models of H\'enon maps}

\author{Ermerson Araujo}
\date{\today}

\address{Ermerson Araujo, Departamento de Matem\'atica, Centro de Ci\^encias, Campus do Pici,
Universidade Federal do Cear\'a (UFC), Fortaleza -- CE, CEP 60440-900, Brasil}
\email{ermersonaraujo@gmail.com}

\renewcommand{\thetheoremletra}{\Alph{theoremletra}}

\begin{abstract}

The purpose of this article is to study the relation between combinatorial equivalence and 
topological conjugacy, specifically how a certain type of combinatorial equivalence implies 
topological conjugacy. We introduce the concept of kneading 
sequences for a setting that is more general than one-dimensional dynamics:
for the two-dimensional toy model family of  H\'enon maps introduced by Benedicks and Carleson, 
we define kneading sequences for their critical lines, and 
prove that these sequences are a complete invariant
for a natural conjugacy class among the toy model family.
We also establish a version of Singer's Theorem for the toy model family.
\end{abstract}

\maketitle


\section{Introduction}

One of the main questions in dynamical systems is to identify when two
systems are `the same', where the term `the same' means 
some type of equivalence between the systems. 
One of the simplest notions of equivalence is {\it topological conjugacy}.
Let $X$ and $Y$ be topological spaces, and let $f:X\to X$ and $g:Y\to Y$ be 
continuous maps. We say that $f,g$ are topologically 
conjugate (or simply {\it conjugate}) if there is a 
homeomorphism $h:X\to Y$ satisfying the conjugacy equation $h\circ f=g\circ h$. 

Topological conjugacy heavily depends on class of maps. For orientation preserving
circle homeomorphisms, Poincar\'e introduced in the late XIX century the notion of {\it rotation numbers},
which essentially characterises all conjugacy classes: if $f:\mathbb{S}^1\to\mathbb{S}^1$
is an orientation preserving 
homeomorphism with irrational rotation number $\rho$, then $f$ is topologically semiconjugate to
the rigid rotation $R_\rho(x)=x+\rho$.
Furthermore, if $f$ contains a dense orbit then $f$ is indeed topologically conjugate to $R_\rho$.

For interval homeomorphisms, topological conjugacy is 
trivial, due to the following fact: every orbit is either periodic 
or converges to a periodic orbit. Now, if we consider interval 
{\em endomorphisms}, then the situation is much richer.
For this class, Milnor and Thurston developed a theory that substitutes the theory of rotation
numbers of Poincar\'e \cite{MT}. Nowadays known as {\it kneading theory}, it 
plays an important role in one-dimensional dynamics. It allows, for example,
to prove that topological entropy varies continuously in the the quadratic family, see \cite{dMvS} for details.
The setting considered by Milnor and Thurston is the following: let $I$ be an interval,
and let $f:I\to I$ be a map with finitely many turning points (i.e. a point $x\in I$ where
$f$ changes monotonicity). For each turning point, they defined a sequence, called {\it kneading sequence}, 
that encodes the trajectory of this point. Then they proved that if two interval maps do not have wandering intervals,
periodic attractors neither intervals of periodic points, then they are topologically conjugate 
if and only if their kneading sequences are the same. 

When we pass to two-dimensional maps, the situation becomes much more complicated.
One of the main maps considered is $H_{a,b}:\mathbb{R}^2\to\mathbb{R}^2$ given by
$H_{a,b}(x,y)=(a-x^2-by,x)$, $a,b\in\mathbb{R}$, known as {\it H\'enon map}.
Over the last years, there has been various attempts to construct a similar theory in this context,
but none of them is yet completely satisfactory. The difficulty comes from two facts: the plane does not possess
a natural order as the interval, and there is no sufficiently nice dynamical notion of critical point. 

In this work we introduce kneading sequences for a particular two-dimensional family.
This family was considered by Benedicks and Carleson as a toy model for the study of H\'enon maps \cite{BC}.
Each of these maps acts on a two-dimensional rectangle via the expression $F(x,y)=(f(x,y), K(x,y))$,
where $f$ is a family of unimodal maps and $K$ is an inverse branch of a Cantor map. We call 
each map of this form a {\em toy model}.
Matheus, Moreira and Pujals proved that, among the toy models, Smale's Axiom A property is 
$C^1$-dense and, on the $C^2$-topology, there exists an open subset satisfying a
Newhouse-like phenomenon \cite{MMP}. These properties indicate that the toy models exhibit 
rich dynamical properties. It is for them that we characterize topological conjugacy in terms
of kneading sequences.

\begin{theoremletra}\label{TheA}
Let $F,G$ be two toy models, and assume that
neither of them has wandering intervals, interval of periodic points nor
weakly attracting periodic points. Then $F,G$ are fiber topologically conjugate 
if and only if they have the same kneading sequences.
\end{theoremletra}

The hardest part in Theorem \ref{TheA} is the reverse implication, that
same kneading sequences lead to fiber topological conjugacy. To construct the conjugacy,
we employ ideas similar to those used in 
one-dimensional dynamics: we define it in preimages of turning points and
then extend to the closure. This approach also leads to the same result for more general toy models,
see Theorem \ref{theorem-generalized}.

Theorem \ref{TheA} has an underlying motivation that we now explain.
For dissipative surface diffeomorphisms, Crovisier and Pujals introduced 
a new class of diffeomorphisms, called {\em strongly dissipative}, whose 
dynamical properties behave as an intermediate level between one-dimensional dynamics
and two-dimensional dynamics \cite{PC}. They proved that this class is open and non-empty (it contains 
H\'enon maps with Jacobian in $(-1/4, 1/4)$), and that every strongly dissipative diffeomorphism of the disc 
(in particular H\'enon maps) is equivalent to a one-dimensional structure (it is semiconjugate to an 
endomorphism defined on a tree). If two strongly dissipative diffeomorphisms on the disc are topologically
conjugate, then their one-dimensional structures are also conjugate. It is unknown
if this one-dimensional structure carries all relevant dynamical information, i.e., if the 
one-dimensional dynamics are conjugate, are the strongly dissipative diffeomorphisms also conjugate?
Theorem \ref{TheA} shows that, on the
presence of a good definition of turning points, we can recover information about the conjugacy class
of the toy model by means of an appropriate combinatorial structure. We hope that the methods
introduced in the proof of Theorem \ref{TheA} will be used to study the conjugacy 
classes of strongly dissipative diffeomorphisms of the disc. 

It is worth mentioning that kneading sequences are not the only method to characterize
topological conjugacy. Based on numerical experiments \cite{cvi}, Cvitanovi\'c introduced 
in \cite{civ2} the concept of {\it pruning fronts} and conjectured that every map 
$H_{a,b}$ in the H\'enon family 
can be understood as a {\it pruned horseshoe}: 
if $F:\mathbb{R}^2\to\mathbb{R}^2$ is Smale's horseshoe map, 
then after pruning (destroying) some orbits of $F$ we construct 
a new map $\widetilde{F}$ that is equivalent, in some sense, to $H_{a,b}$. 
This is known as the Pruning Front Conjecture. 
Mendonza proved that the Pruning Front Conjecture holds in an open set of
the parameter space \cite{men1}, and Ishii proved it for the Lozi family \cite{ishii1}.
For more detail on the theory of pruning fronts and its relationship with kneading theory,
see \cite {carvalho2,carvalho3}. 

On a different direction of kneading sequences, we also establish for the toy models
a version of the classical Singer theorem. Let us explain this. Singer proved that, 
for $C^3$ interval maps with negative Schwarzian derivative, the basin of any
attracting periodic point contains either a critical point or a boundary point of the interval.
Using that the toy models do preserve horizontal lines, 
we prove the following theorem, which is a version of Singer's theorem for the toy models.

\begin{theoremletra}\label{thesinger222}
Let $F(x,y)=(f(x,y), K(x,y))$ be a toy model.
If each interval map $f(\cdot,y)$ has negative Schwarzian 
derivative, then the closure of the immediate 
basin of any strongly attracting periodic orbit contains either a 
point of the critical line or a boundary point of the rectangle.
\end{theoremletra}

This paper is organized as follows. Section \ref{section-toy-model} introduces in full details the 
toy models and consider their kneading sequences. It also states the main results, establishes 
some basic properties of the toy models, and discusses examples.
Section \ref{section-proofs} contains the proofs of the main results stated in Section \ref{section-toy-model}.

\vspace{.2cm} 

\medskip
\noindent
{\bf Acknowledgements.} This article is part of my Ph.D thesis at IMPA
under the supervision of Enrique Pujals. I deeply thank him for his constant
support, patience, encouragement and availability. I also would like to thank
Yuri Lima for the carefully reading and several helpful comments. This
work was supported by CNPq and CAPES.


\section{Toy models and kneading sequences}\label{section-toy-model}

Results in one-dimensional dynamics heavily rely on the total order of the line.
For example, given two multimodal interval maps, the preimages
of the turning points divide the interval from left to right, and if these preimages
are combinatorially equivalent then there is a monotone bijection
between them. The family of two-dimensional maps we consider also inherits this property:
as we will see below, the one-dimensional order structure
is preserved along ``unstable leaves''.


\subsection{The Toy Model}

Given a continuous interval map $f:[a,b]\to [a,b]$, we call it
a {\it unimodal map} if $f(a)=f(b)=a$ and if there exists a
unique point $c\neq a,b$, called {\em turning point}, for which $f$ is
increasing on $[a,c]$ and decreasing on $[c,b]$. We consider two-dimensional maps 
with fiber dynamics given by unimodal maps. We follow \cite{MMP} for the description.

\medskip
\noindent
{\sc One-Parameter Family of Unimodal Maps:} Consider a one-parameter family
\[
f(y):[-1,1]\to[-1,1],\ \ y\in[0, 1],
\]
such that:
\begin{enumerate}[$\circ$]
\item $y\mapsto f(y)$ is continuous,
\item each $f(y)$ is a unimodal map with $f(y)(-1)=f(y)(1)=-1$ and $0$ as a turning point.
\end{enumerate}

\vspace{.2cm}
Fix $a<b$ in $(0,1)$.

\medskip
\noindent
{\sc Cantor Map:} A {\em Cantor map} is a differentiable map $k:[0,a]\cup[b,1]\to[0,1]$ such that
$k(0)=k(1)=0$, $k(a)=k(b)=1$ and $\vert k'\vert>\gamma>1$.
Each Cantor map defines a {\em Cantor set} $\mathscr{K}^F$ given by
\[
\mathscr{K}^F=\displaystyle\bigcap_{n\geq0}k^{-n}([0,a]\cup[b,1]).
\]

Let $K_+=(k\restriction_{[0, a]})^{-1}$ and $K_-=(k\restriction_{[b, 1]})^{-1}$ be the inverse branches of $k$,
and define $K(x,y)$ by
$$
K(x,y)= \left\{
\begin{array}{rl}
K_+(y), & \textrm{if } x>0   \\
K_-(y), & \textrm{if } x<0.
\end{array}
\right.
$$

\medskip
\noindent
{\sc Toy Model:} A {\it toy model} is a map $F$ of the form
\[
\begin{array}{crcl}
F \ : & \! ([-1,1]\backslash\{0\})\times [0,1]& \! \longrightarrow  
& \! [-1,1]\times[0,1] \\
        & \! (x,y) & \! \longmapsto      & \! (f(y)(x), K_{sign(x)}(y)).
\end{array}
\]

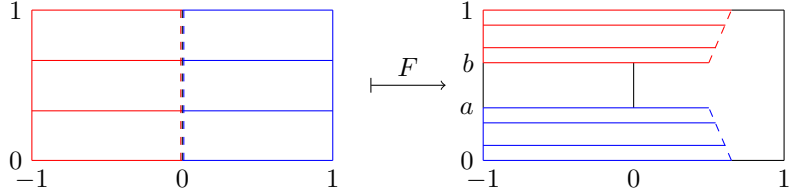
\begin{figure}[H]\centering
	\begin{tikzpicture}
	\draw[color=red] (0,0) -- (2,0) ;
	\draw[color=red] (0,0) -- (0,2) ;
	\draw[color=red] (0,2) -- (2,2) ;
	\draw[color=red] (0,1.33) -- (2,1.33) ;
	\draw[color=red] (0,0.66) -- (2,0.66) ;
	\draw[color=red, dashed] (1.98,2) -- (1.98,0) ;
	\draw[dashed] (2,2) -- (2,0) ;
	\draw[color=blue, dashed] (2.02,2) -- (2.02,0) ;
	\draw[color=blue] (2,0) -- (4,0) ;
	\draw[color=blue] (2,2) -- (4,2) ;
	\draw[color=blue] (4,2) -- (4,0) ;
	\draw[color=blue] (2,1.33) -- (4,1.33) ;
	\draw[color=blue] (2,0.66) -- (4,0.66) ;
	\fill[black] (5,1) circle (0.mm) node[above]{$F$};
	\fill[black] (2,0) circle (0.mm) node[below]{$0$};
	\fill[black] (0,0) circle (0.mm) node[below] {$-1$};
	\fill[black] (4,0) circle (0.mm) node[below] {$1$};
	\fill[black] (0,2) circle (0.mm) node[left] {$1$};
	\fill[black] (0,0) circle (0.mm) node[left] {$0$};
	\draw[|->] (4.5,1) -- (5.5,1) ;
	\draw[] (6,0.7) -- (6,1.3) ;
	\draw[] (9.3,0) -- (10,0) ;
	\draw[] (8,0.7) -- (8,1.3) ;
	\draw[] (9.3,2) -- (10,2) ;
	\draw[] (10,2) -- (10,0) ;
	\draw[color=red] (6,2) -- (9.3,2) ;
	\draw[color=red] (6,1.3) -- (6,2) ;
	\draw[color=red] (6,1.3) -- (9,1.3) ;
	\draw[color=red, dashed] (9,1.3) -- (9.3,2) ;
	\draw[color=red] (6,1.8) -- (9.22,1.8) ;
	\draw[color=red] (6,1.5) -- (9.09,1.5) ;
	\draw[color=blue] (6,0) -- (9.3,0) ;
	\draw[color=blue] (6,0.7) -- (9,0.7) ;
	\draw[color=blue] (6,0) -- (6,0.7) ;
	\draw[color=blue, dashed] (9.3,0) -- (9,0.7) ;
	\draw[color=blue] (6,0.2) -- (9.22,0.2) ;
	\draw[color=blue] (6,0.5) -- (9.09,0.5) ;
	\fill[black] (8,0) circle (0.mm) node[below]{$0$};
	\fill[black] (6,0) circle (0.mm) node[below] {$-1$};
	\fill[black] (10,0) circle (0.mm) node[below] {$1$};
	\fill[black] (6,2) circle (0.mm) node[left] {$1$};
	\fill[black] (6,0) circle (0.mm) node[left] {$0$};
	\fill[black] (6,0.7) circle (0.mm) node[left] {$a$};
	\fill[black] (6,1.3) circle (0.mm) node[left] {$b$};
\end{tikzpicture}
\caption{Depiction of the dynamics of a toy model $F$.}
\end{figure}

The map $F$ is not defined on the vertical line $\{0\}\times [-1,1]$, and this 
may introduce serious difficulties to analyse $F$ and apply compactness arguments.
To bypass this, we duplicate each $(0,y)$ to two points $(0^-,y)$ and $(0^+,y)$
and define $F$ on them by $F(0^\pm,y)=(f(y)(0),K_\pm(y))$.

\medskip
\noindent
{\sc Critical Line of $F$:} It is the set 
\[
\li_c(F)=\{(0^\pm, y):y\in[0,1]\}
\] 
equal to the disjoint union of two vertical segments $\{0^-\}\times [0,1]$ and $\{0^+\}\times [0,1]$.
Each element of $\li_c(F)$ is called a {\it turning point} of $F$.

Therefore $F$ is now defined on $\left(([-1,1]\backslash\{0\})\times[0,1]\right)\cup\li_c(F)$,
call this set $\text{Dom}(F)$. We define a natural topology on $\text{Dom}(F)$ as follows:
\begin{enumerate}[$\circ$]
\item If $(x,y)\in\text{Dom}(F)\backslash\li_c(F)$, then 
the neighborhoods of $(x,y)$ are the standard Euclidean neighborhoods.
\item If $x=0^-$ then the neighborhoods
of $(x,y)$ are $\mathcal U\cap ([-1,0^-]\times [0,1])$, where $\mathcal U$ is  a standard
Euclidean neighborhood of $(0,y)$.
\item If $x=0^+$ then the neighborhoods
of $(x,y)$ are $\mathcal U\cap ([0^+,1]\times [0,1])$, where $\mathcal U$ is  a standard
Euclidean neighborhood of $(0,y)$.
\end{enumerate}
We remark that, although $F$ is now defined on a compact set, it is {\em not} continuous.


\subsection{Notations and preliminaries}

For $y\in[0,1]$, write $f_-(y)$ for the restriction $f(y)\restriction_{[-1, 0^-]}$ and
$f_+(y)$ for the restriction $f(y)\restriction_{[0^+, 1]}$, see Figure \ref{f(y)}. 

\begin{figure}[H]\centering
\begin{tikzpicture}[scale=1.5]
      \draw[] (1,-1) -- (-1,-1) node[below] {$-1$};
	  \draw[] (1,1) -- (1,-1) node[below] {$1$};
	  \draw[] (1,1) -- (-1,1) node[left] {$1$};
	  \draw[] (-1,1) -- (-1,-1) node[left]{$-1$};
	  \draw[dashed] (1,0) -- (-1,0) node[left]{$0$};
      \draw[dashed] (0,1) -- (0,-1) node[below] {$0$};
      \draw[color=red] (0,0.7) parabola (-1,-1) ;
	    \draw[color=blue] (0,0.7) parabola (1,-1) ;
	\fill[red] (-0.7,0.4) circle (0.mm) node[above] {$f_-(y)$};
	\fill[blue] (0.7,0.4) circle (0.mm) node[above] {$f_+(y)$};
\end{tikzpicture}
\caption{The restriction maps $f_-(y)$ and $f_+(y)$.}
\label{f(y)}
\end{figure}
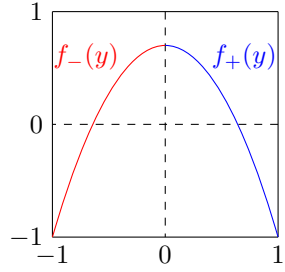
Also, let $F_-$ and $F_+$ be the restrictions of $F$ defined by
\[
\begin{array}{crcl}
F_- \ : & \! [-1,0^{-}]\times [0,1]& \! \longrightarrow  
& \! ([-1,0^{-}]\cup(0^{+},1])\times[0,1] \\
        & \! (x,y) & \! \longmapsto      & \! (f_-(y)(x), K_-(y))
\end{array}
\]
and
\newline
\[
\begin{array}{crcl}
F_+ \ : & \! [0^+,1]\times [0,1]& \! \longrightarrow  
& \! ([-1,0^{-})\cup[0^{+},1])\times[0,1] \\
        & \! (x,y) & \! \longmapsto      & \! (f_+(y)(x), K_+(y)).
\end{array}
\]
Every orbit of $F$ is therefore obtained by compositions of $F_-$ and $F_+$:
for each $(x,y)\in\text{Dom}(F)$ there is a sequence 
$j(x,y)=(j_1j_2\cdots j_m\cdots)\in\{-,+\}^\N$
such that
\[
F^{m+1}(x,y)=F_{j_{m+1}}(x_m,y_m),\ \ m\geq 0,
\]
where 
\[y_m=(K_{j_m}\circ\cdots\circ K_{j_1})(y)=:y_{j_m\cdots j_1} \text{ and }
x_m=f_{j_m}(y_{j_{m-1}\cdots j_1})\circ\cdots\circ
f_{j_2}(y_{j_1})\circ f_{j_1}(y)(x).
\]
Let $I(y):=([-1,0^-]\cup[0^+,1])\times\{y\}$, and
call $J\subset I(y)$ an {\it interval} if there are
$x_1,x_2\in[-1,0^-]$ or $x_1,x_2\in[0^+,1]$ such that
\[
J=\{x\in[-1,0^-]\cup[0^+,1]\;:\; x_1< x< x_2\}\times\{y\}.
\]
We will sometimes denote $J$ as above by $[(x_1,x_2),y]$. If 
\[
J=\{(x,y)\in[-1,0^-]\cup[0^+,1]\;:\; x_1\leq x\leq x_2\}\times\{y\},
\]
then we write $J=[[x_1,x_2],y]$. 
Given a periodic point $(p,q)$ of $F$, set
\[
\mathscr{B}(p,q)=\{(x,y): F^\ell(x,y)\to\mathscr{O}(p,q)\; \textrm{as}\;
\ell\to\infty\}. 
\]
We call $(p,q)$ a {\it weakly attracting periodic point} if
$\mathscr{B}(p,q)$ contains an interval. If $\mathscr{B}(p,q)$
contains an open subset, then call $(p,q)$ a
{\it strongly attracting periodic point}. 
When $\mathscr{O}_F(p,q)\cap\li_c(F)=\emptyset$, then
the immediate basin $\mathscr{B}_0(p,q)$ of $\mathscr{O}(p,q)$ is the union of the 
connected components of $\mathscr{B}(p,q)$ that contain of the set
from $\{(p,q), F(p,q),\ldots, F^{n-1}(p,q)\}$, 
where $n$ is the period of $(p,q)$. Hence $\mathscr{B}_0(p,q)\cap\li_c(F)=\emptyset$,
and so there is a sequence $j_1\cdots j_n\in\{-,+\}^n$
such that for each $(x,y)\in B(p,q)$ we have
\[
\left(f_{j_n}(y_{j_{n-1}\cdots j_1})\circ\cdots\circ
f_{j_2}(y_{j_1})\circ f_{j_1}(y)(x), 
y_{j_n\cdots j_1}\right)\in B(p,q),
\]
where $B(p,q)\subset\mathscr{B}_0(p,q)$ is the
component containing $(p,q)$.
Clearly, if $(p,q)$ is a strongly attracting 
periodic point, then it is also a weakly attracting periodic point. 

Given $y\in[0,1]$, an interval $J\subset I(y)$ is called an
{\it interval of periodic points} if every $(x,y)\in J$ is a
periodic point of $F$.

\medskip
\noindent 
{\sc Wandering Interval:} An interval $J\subset{\rm Dom}(F)$ is called
{\em wandering} if 
\begin{enumerate}[$\circ$]
\item $F^n(J)\cap F^m(J)=\emptyset$, $\forall n\neq m$;
\item $F^n(J)\cap\li_c(F)=\emptyset$, $\forall n\geq 0$.
\end{enumerate}

\vspace{.2cm}
A point $(x,y)$ is called {\em non-wandering} if
every neighborhood $V\subset\text{Dom}(F)$ containing $(x,y)$ 
has an iterate $\ell\geq1$ such that $F^\ell(V)\cap V\neq\emptyset$.
The {\it non-wandering set} $\Omega(F)$ is the set of all non-wandering points. 

It is well known that if $f:I\to I$ is an unimodal map without
intervals of periodic points nor attracting periodic points,
then the set of pre-images of the turning point is dense in $\Omega(F)$.
We establish the same for toy models. Define
\[
\cc(F)=\left\{(x,y)\in\text{Dom}(F):\exists\; \ell\geq0 
\;\textrm{such that}\; F^\ell(x,y)\in\li_c(F)\right\}.
\]

\begin{proposition}\label{nonwanprop}
Let $F$ be a toy model, and assume that $F$ has no weakly attracting periodic point nor intervals of
periodic points. Then $\Omega(F)\subset\overline{\cc(F)}$.
\end{proposition}

\begin{proof}
The proof is by contradiction.
Suppose that $(x,y)\in\Omega(F)\backslash\overline{\cc(F)}$.
Let $V$ be a neighborhood of $(x,y)$ such that 
$F^n(V)\cap\li_c(F)=\emptyset$ for all $n\geq0$.
Without loss of generality, we can assume that $V=I\times J$, 
where $x\in I$, $y\in J$, and $I,J$ are open intervals. 
Fix $m\geq1$ with $F^m(V)\cap V\neq\emptyset$, and put 
\[
\widehat{L}=\displaystyle\bigcup_{n\geq0} F^n(V).
\]
Since, for each $r=0,\ldots,m-1$, the union 
$\bigcup_{j\geq0} F^{jm+r}(V)$ is connected, the set
$\widehat{L}$ has finitely many connected components.
Set $\widehat{L}=L_1\sqcup\cdots\sqcup L_s$, where $L_1,\ldots,L_s$ are connected.
Using that $F(\widehat{L})\subset\widehat{L}$, for each $i$ 
there exists $j$ such that $F(L_i)\subset L_j$. 
Let $L$ be the connected component containing $(x,y)$.
We have $F^m(L)\cap L\neq\emptyset$, hence $F^m(L)\subset L$.
Take $\ell=\min\{d\geq1;\:F^d(L)\subset L\}$.
Hence, we have that $L$ is a domain such that $F^\ell(L)\subset L$,
$F^i(L)\cap\li_c(F)=\emptyset$ for $i=0,1,\ldots,\ell-1$, and 
$F^i(L)\cap F^j(L)=\emptyset$
for $i,j=0,1,\ldots \ell-1$ with $i\neq j$. 
Therefore there exists a sequence $j_0j_1\cdots j_{\ell-1}\in\{-,+\}^\ell$
such that for any $(a,b)\in L$ and $n\geq0$ 
it holds $F^{n\ell}(a,b)=(a_{n\ell},b_{(j_{\ell-1}\cdots j_1j_0)^n})$, where
$b_{(j_{\ell-1}\cdots j_1j_0)^n}= \left(K_{j_{\ell-1}}\circ\cdots\circ K_{j_0}\right)^n(b)$
and 
\[
a_{n\ell}=f_{j_{\ell-1}}(b_{j_{\ell-2}\cdots j_0(j_{\ell-1}\cdots j_0)^{n-1}})
\circ\cdots\circ
f_{j_{\ell-1}}(b_{j_{\ell-2}\cdots j_0})\circ\cdots\circ 
f_{j_1}(b_{j_0})\circ f_{j_0}(b)(a).
\] 
Let $\pi_2$ be the projection on the second coordinate, $\pi_2(x,y)=y$.
Since $K:=K_{j_{\ell-1}}\circ\cdots\circ K_{j_0}$ 
is a contraction, there exists $w\in\pi_2(\overline{L})$ such that $K(w)=w$.
We claim that $w=y$. To see that,
for each $n\geq1$ take $V_n:=I_n\times J_n\varsubsetneqq I\times J$ 
neighborhood of $(x,y)$ such that 
$I_{n+1}\times J_{n+1}\varsubsetneqq I_n\times J_n$  with 
$|I_n|\to0$ and $|J_n|\to0$ as $n\to\infty$.
As before, for each $n\geq1$ there are
$(x_n,y_n)\in V_n$ and $k_n\geq1$ such that 
$F^{k_n}(x_n,y_n)\in V_n$ and 
$F^{k_n}(x_n,y_n)\to(x,y)$ as $n\to\infty$. 
We may suppose that $k_n\to\infty$ as $n\to\infty$.
Furthermore, by the construction of $L$, 
we have $k_n=r_n\ell$ for some $r_n\geq1$.
We already know that $K$ is a contraction, 
hence there is $\lambda<1$ such that 
\[
|K(x)-K(y)|\leq\lambda|x-y|,\ \forall x,y\in[0,1].
\]
Note that
\begin{eqnarray*}
|K^{r_n}(y_n)-w|&\leq&|K^{r_n}(y_n)-K^{r_n}(y)|+|K^{r_n}(y)-w| \\ 
      					&\leq&\lambda^{r_n}|y_n-y|+|K^{r_n}(y)-w|\to 0,
\end{eqnarray*}
therefore
\[
\pi_2[F^{k_n}(x_n,y_n)]=K^{r_n}(y_n)\to w.
\]
Using that $F^{k_n}(x_n,y_n)\to(x,y)$ as $n\to\infty$, we conclude that $w=y$, as claimed.

Now, for every $(v,y)\in L\cap I(y)$ and $n\geq1$ we get
\[
F^{m\ell}(v,y)=\left(\left(f_{j_{\ell-1}}(y_{j_{\ell-2}\cdots j_0})
\circ\cdots\circ f_{j_0}(y)\right)^m(v),y\right)\in L\cap I(y),
\]
where $f_{j_{\ell-1}}(y_{j_{\ell-2}\cdots j_0})
\circ\cdots\circ f_{j_0}(y)$ is a strictly monotone map. 
Without loss of generality, we assume that
$f_{j_{\ell-1}}(y_{j_{\ell-2}\cdots j_0})\circ\cdots\circ f_{j_0}(y)$ is strictly increasing.
Let $\{I_i\}_i$ be the connected components of $L\cap I(y)$.

\medskip
\noindent
{\sc Claim:} There exists $i_0$ such that $F^\ell(I_{i_0})\subset I_{i_0}$.

\medskip
\noindent
{\em Proof of the claim.} Assume, by contradiction, that it does not hold.
Since the composition $f_{j_{\ell-1}}(y_{j_{\ell-2}\cdots j_0})\circ\cdots\circ f_{j_0}(y)$ 
is strictly increasing, for every $i$ there exists $\sigma(i)$ such that
$F^\ell(I_i)\subset I_{\sigma(i)}$ and $\sup I_i<\inf I_{\sigma(i)}$.
Fixed any $k$, this implies that the sequence $\{F^{\ell n}(I_k)\}_{n\geq0}$
is formed by disjoint intervals satisfying $\sup F^{\ell n}(I_k)<
\inf F^{\ell(n+1)}(I_k)$ and $|F^{\ell n}(I_k)|\to0$ as $n$ goes to infinity.
This implies that there is $(v_0,y)$ such that $F^\ell(v_0,y)=(v_0,y)$ and
$F^{\ell n}(I_k)\to(v_0,y)$, contradicting the non-existence of 
weakly attracting periodic point. The claim is proved.

\medskip
Now we complete the proof of the proposition. Let $i_0$ be given by the claim. The
map $F^\ell:I_{i_0}\to I_{i_0}$ is strictly increasing, hence either $I_{i_0}$ contains 
an interval of periodic points for $F^\ell$ , or some open subinterval 
of $I_{i_0}$ converges to a single periodic point. Both possibilities contradict the 
underlying assumptions on $F$.
\end{proof}


\subsection{Kneading sequences and main results}

Let $\pi_1$ be the projection on first coordinate, $\pi_1(x,y)=x$. Let $F$ be a toy model. 
Consider the alphabet $\mathcal{A}=\{L, 0^-,0^+, R\}$. The letter $L$ means ``left''
and $R$ means ``right''.

\medskip
\noindent
{\sc Address of a point:} The {\em address} of a point $(x,y)\in\text{Dom}(F)$ is 
the letter $t_F(x,y)\in\mathcal A$ defined by:
$$
t_F(x,y)= \left\{
\begin{array}{ll}
L     & \textrm{, if }  \pi_1(x,y)<0     \\
0^- & \textrm{, if }  \pi_1(x,y)=0^- \\
0^+ & \textrm{, if }  \pi_1(x,y)=0^+ \\
R     & \textrm{, if }  \pi_1(x,y)>0.
\end{array}
\right.
$$

\medskip
\noindent
{\sc Itinerary of a point:} The {\em itinerary} of a point $(x,y)\in\text{Dom}(F)$
is the sequence $T_F(x,y)\in\mathcal A^{\{0,1,2,\ldots\}}$ defined by 
\[
T_F(x,y)=(t_F(x,y),t_F(F(x,y)),\ldots,t_F(F^n(x,y)),\ldots).
\]

\medskip
\noindent
{\sc Kneading Sequences:} The {\em kneading sequences} of $F$ are all the itineraries
of elements of $\li_c(F)$. We denote this set of itineraries by $\mathcal{K}_F(\li_c)$. 

\medskip
In the following proposition, we construct conjugacies between the Cantor sets appearing in 
the second coordinate of two toy models. The proof is standard, but we include it for completeness.

\begin{proposition} \label{prop1}
Let $F(x,y)=(f(y)(x), K^F_{{\rm sign}(x)}(y))$ and 
$G(x,y)=(g(y)(x),\\ K^G_{{\rm sign}(x)}(y))$ 
be two toy models. There exists a strictly increasing continuous 
map $\psi:[0,1]\to[0,1]$ such that  
\[
\psi\circ K^F_\pm=K^G_\pm\circ\psi.
\] 
Moreover, $\psi\restriction_{\mathscr{K}^F}$ is independent of the choice of $\psi$.
\end{proposition}

\begin{proof}
Let $k^F:[0,a^F]\cup[b^F,1]\to[0,1]$ and $k^G:[0,a^G]\cup[b^G,1]\to[0,1]$
the Cantor maps of $F$ and $G$.
If $\mathscr{K}^F$ and $\mathscr{K}^G$ are the Cantor sets generated by $k^F$ and $k^G$ respectively, then
there is a unique strictly increasing bijection 
$\psi:\mathscr{K}^F\to\mathscr{K}^G$ such that $\psi\circ k^F=k^G\circ\psi$. To see this,
observe that if we order $\{0,1\}^\N$ lexicographically, then $\mathscr{K}^F$ and $\mathscr{K}^G$
are isomorphic to $\{0,1\}^\N$ through an order preserving bijection. 
Hence, for each $x\in\mathscr{K}^F$ we take $\psi(x)$ to be
the unique point of $\mathscr{K}^G$ with the same symbolic representation
as $x$ in $\mathscr{K}^F$.
Now we extend $\psi$ continuously to gaps of $\mathscr{K}^F$. Take any 
strictly increasing homeomorphism $h:(a^F, b^F)\to(a^G, b^G)$.
Denote by $k^{F,G}_+$ the restriction ${k^{F,G}}\restriction_{[0, a^{F,G}]}$ and
by $k^{F,G}_-$ the restriction ${k^{F,G}}\restriction_{[b^{F,G}, 1]}$.
Given a gap $J\subset[0,1]$ of $\mathscr{K}^F$, there exists a unique $n\geq1$ and 
a sequence $j_1\cdots j_n\in\{-,+\}^n$ such that
$(k^F_{j_n}\circ\cdots\circ k^F_{j_1})(J)=(a^F, b^F)$. 
Thus, for each $x\in J$ we define $\psi(x)$ by
\[
\psi(x):=[(k^G_{j_1})^{-1}\circ\cdots\circ (k^G_{j_n})^{-1}]
\circ h\circ [k^F_{j_n}\circ\cdots\circ k^F_{j_1}](x).
\]
Clearly $\psi:[0,1]\to[0,1]$ is a strictly increasing continuous 
map with $\psi\circ K^F_\pm=K^G_\pm\circ\psi$.
By construction, $\psi\restriction_{\mathscr{K}^F}$ does not depend on the choice of $\psi$.   
\end{proof}

\begin{definition}\label{deff3}
Consider two toy models $F(x,y)=(f(y)(x), K^F_{{\rm sign}(x)}(y))$ and 
$G(x,y)=(g(y)(x),K^G_{{\rm sign}(x)}(y))$, and let $\psi:[0,1]\to[0,1]$ be the 
map given by Proposition \ref{prop1}. 
We say that {\em $F$ and $G$ have the same kneading sequences}, and 
write $\mathcal{K}_F(\li_c)=\mathcal{K}_G(\li_c)$,
if 
\[
T_F(0^\pm,y)=T_G(0^\pm, \psi(y)),\ \forall y\in[0,1].
\]
\end{definition}

Observe that the above definition does not actually depend on the choice of $\psi$.
Indeed, the invariant set defined by $F$ is 
$\left([-1,0^-]\cup[0^+,1]\right)\times\mathscr{K}^F$, as we know that in $\mathscr{K}^F$
the map $\psi$ is uniquely defined. We are now ready to state 
the main tool to be used in the proof of Theorem A.

\begin{theorem}\label{maintheorem}
Let $F, G$ be two toy models, and assume that $\mathcal{K}_F(\li_c)=\mathcal{K}_G(\li_c)$.
Then there exists a bijective map 
$H:\cc(F)\to\cc(G)$ such that:
\begin{enumerate}[i\,\,\,]
\item[{\rm (i)}] $H\circ F=G\circ H$ on $\cc(F)\backslash\li_c(F)$.
\item[{\rm (ii)}] $H(\cc(F)\cap I(y))\subset\cc(G)\cap I(\psi(y))$ for all $y\in[0,1]$.
\item[{\rm (iii)}] $H\restriction_{\cc(F)\cap I(y)}$ is strictly increasing for all $y\in[0,1]$.
\end{enumerate}
\end{theorem}

The proof of Theorem \ref{maintheorem} will be given in the next section.
When the map $H$ given by Theorem \ref{maintheorem} exists, we say that
$F$ and $G$ are {\em combinatorially equivalent}.
 
\begin{definition}\label{defconj}
Consider two toy models $F(x,y)=(f(y)(x), K^F_{{\rm sign}(x)}(y))$ and 
$G(x,y)=(g(y)(x),K^G_{{\rm sign}(x)}(y))$, and let $\psi:[0,1]\to[0,1]$ be the 
map given by Proposition \ref{prop1}.
We say that $F$ and $G$ are {\em fiber topologically semiconjugate} 
if there is a continuous surjective map $H:{\rm Dom}(F)\to{\rm Dom}(G)$
which is increasing on intervals, such that
$H\circ F=G\circ H$ and $H(I(y))\subset I(\psi(y))$ for all $y\in [0,1]$.
If $H$ is a homeomorphism, then we say that $F$ and $G$ are {\em fiber topologically conjugate}.
\end{definition}

If $H$ conjugates $F$ and $G$ as above, then $H(0^\pm,y)=(0^\pm,\psi(y))$ for all $y\in[0,1]$.
In particular, $H(\li_c(F))=\li_c(G)$.
The following proposition states that kneading sequences are a topological invariant
for the notion introduced in Definition \ref{defconj}.

\begin{proposition}\label{pprop}
Let $F, G$ be two toy models. If $F,G$ are fiber
topologically conjugate then $\mathcal{K}_F(\li_c)=\mathcal{K}_G(\li_c)$.
\end{proposition}

\begin{proof}
Let $H:{\rm Dom}(F)\to{\rm Dom}(G)$ be conjugacy between $F$ and $G$.
We will show that $T_F(0^\pm,y)=T_G(0^\pm,\psi(y))$ for all $y\in[0,1]$.
Fix $y\in[0,1]$.
Clearly $t_F(0^\pm,y)=t_G(0^\pm,\psi(y))$.
Since $F,G$ are fiber topologically conjugate,
\[
H[F^n(0^\pm,y)]=G^n[0^\pm,\psi(y)],\ \forall n\geq 1.
\]
Since $H$ is strictly increasing on intervals,
\[
t_F(F^n(0^\pm,y))=t_G(G^n[0^\pm,\psi(y)]),\ \forall n\geq 1.
\]
Therefore $T_F(0^\pm,y)=T_G(0^\pm,\psi(y))$. 
\end{proof}

Notice that the above proposition implies the first part of Theorem \ref{TheA}.
Moreover, within the notion of topological conjugacy given by Definition \ref{defconj},
Theorem \ref{TheA} establishes that kneading sequences are a complete topological invariant.
It would be natural to find, inside the class of toy models,
a family of representatives for each  combinatorial equivalence class. 
To better understand this question, we remind a classical result
in one-dimensional dynamics. Consider the {\em quadratic family} $\{f_\mu\}_{0\leq \mu\leq 4}$
of maps $f_\mu:[0,1]\to[0,1]$ defined by $f_\mu(x)=\mu x(1-x)$.
This family is {\em universal} for unimodal maps in the following sense: if $g:[0,1]\to[0,1]$
is a unimodal map then $g$ is combinatorially equivalent to $f_\mu$ for some $\mu=\mu(g)$.
Using this analogy, it is natural to find a universal family for toy models. 
At the moment, it is not clear how to construct such a family, left alone its existence. 

\vspace{.2cm}
Let $F$ be a toy model.

\medskip
\noindent 
{\sc Wandering Domain:} A domain $V\subset{\rm Dom}(F)$ is called {\em wandering} if 
\begin{enumerate}[$\circ$]
\item $F^n(V)\cap F^m(V)=\emptyset$, $\forall n\neq m$;
\item $F^n(V)\cap\li_c(F)=\emptyset$, $\forall n\geq 0$.
\end{enumerate}

\vspace{.2cm} 
For $y\in[0,1]$ and $n\geq1$, let
\[
\cc_n^F(y):=\left\{(x,y)\in\text{Dom}(F)\cap I(y):
\;F^\ell(x,y)\in\li_c(F) 
\;\textrm{for some}\; 0\leq\ell\leq n-1\right\}.
\]
Consider the map $\phi_n^F:[0,1]\to2^{\text{Dom}(F)}$ defined by 
$\phi_n^F(y):=\cc_n^F(y)$. Under some regularity of this function,
we can state a weaker version of Theorem \ref{TheA}.

\begin{theorem}\label{maintheorem2}
Let $F,G$ be two toy models. Assume that
$F,G$ have no wandering domain and that $\mathcal{K}_F(\li_c)=\mathcal{K}_G(\li_c)$. 
If the family $\{\phi_n^G\}_{n\geq 1}$ is equicontinuous
then $F,G$ are fiber topologically semiconjugate.
\end{theorem}

Above, equicontinuity means the following: given $\varepsilon>0$ there is 
$\delta>0$ such that for all $y_1,y_2\in[0,1]$ with $|y_1-y_2|<\delta$ it holds 
$d_\mathcal{H}(\cc_n^G(y_1),\cc_n^G(y_2))<\varepsilon$ for {\em all} $n\geq1$,
where $d_\mathcal{H}$ is the Hausdorff distance. 
At the moment, we observe that the assumptions that $G$ has no wandering domain and that
$\{\phi_n^G\}$ is equicontinuous imply that $G$ has no wandering intervals, therefore
Theorem \ref{maintheorem2} follows from Theorem \ref{TheA}. Nevertheless,
we will give an independent proof Theorem \ref{maintheorem2} in the next section.

\subsection{Examples}

We now give some examples of toy models, and discuss the
hypotheses required in Theorem \ref{TheA} and Theorem \ref{maintheorem2}. 

\begin{example}
Let $f:[-1,1]\to[-1,1]$ be the {\em tent map}, i.e.
$f(x)=1+2x$ for $x\in [-1,0]$ and $f(x)=1-2x$ for $x\in[0,1]$. Let $k:[0,1]\to[0,1]$ be the Cantor 
map defined by $k(y)=3y$ if $y\in[0,1/3]$ and $k(y)=3-3y$ if $y\in[2/3,1]$. 
The map $F(x,y):=(f(x), K(x,y))$ satisfies the hypotheses of Theorems \ref{TheA}
and \ref{maintheorem2}.
\end{example}

\begin{example}
Let $f:[-1,1]\to[-1,1]$ be a unimodal map with $0$ as turning point and containing
a wandering interval $I_0$. Let $k$ be the Cantor map of example 1. 
The toy model $F(x,y):=(f(x), K(x,y))$ has wandering intervals $I_0\times\{y\}$.
\end{example}

\begin{example}
Let $q,f,g:[-1,1]\to[-1,1]$ be three unimodal maps defined by
$$
q(x)=-x^2,\ \ 
f(x)= 
\left\{
\begin{array}{rl}
x   & \textrm{, if } x\leq0 \\
-x  & \textrm{, if } x\geq0
\end{array}
\right.,
\ \ g(x)= \left\{
\begin{array}{rl}
-\sqrt{-x} & \textrm{, if } x\leq0 \\
-\sqrt{x}  & \textrm{, if } x\geq0. 
\end{array}
\right.
$$

\begin{figure}[H]\centering
  \begin{tikzpicture}
	\draw[] (0,0) -- (3,0) ;
	\draw[] (0,0) -- (0,3) ;
	\draw[] (0,3) -- (3,3) ;
	\draw[] (3,0) -- (3,3) ;
	\draw[dashed] (0,0) -- (3,3) ;
	\draw[dashed] (0,3) -- (3,0) ;
	\draw[color=red,thick] (1.5,1.5) parabola (0,0) ;
	\draw[color=red,thick] (1.5,1.5) parabola (3,0) ;
	\fill[red,thick] (2.3,1.3) circle (0.mm) node[above] {$q$};
	\draw[] (4,0) -- (7,0) ;
	\draw[] (4,0) -- (4,3) ;
	\draw[] (4,3) -- (7,3) ;
	\draw[] (7,0) -- (7,3) ;
	\draw[dashed] (5.5,1.5) -- (7,3) ;
	\draw[dashed] (4,3) -- (5.5,1.5) ;
	\draw[blue,thick] (4,0) -- (5.5,1.5);
	\draw[blue,thick] (7,0) -- (5.5,1.5);
	\fill[blue,thick] (6.3,1.3) circle (0.mm) node[above] {$f$};
	\draw[] (8,0) -- (11,0) ;
	\draw[] (8,0) -- (8,3) ;
	\draw[] (8,3) -- (11,3) ;
	\draw[] (11,0) -- (11,3) ;
	\draw[dashed] (8,0) -- (11,3) ;
	\draw[dashed] (8,3) -- (11,0) ;
	\coordinate (l) at (8,0);
  \coordinate (m) at (9.5,1.5);
	\draw[color=purple,thick] (l) to [bend right=30] (m);
	\coordinate (a) at (11,0);
  \coordinate (b) at (9.5,1.5);
	\draw[color=purple,thick] (a) to [bend left=30] (b);
	\fill[purple,thick] (10.3,1.3) circle (0.mm) node[above] {$g$};
\end{tikzpicture}
\end{figure}
Let $k$ be the Cantor map defined in example 1, and define the toy models
$Q(x,y)=(q(x), K(x,y))$, $F(x,y)=(f(x), K(x,y))$, and $G(x,y)=(g(x), K(x,y))$. Then
\[
\mathcal{K}_Q(\li_c)=\mathcal{K}_F(\li_c)=\mathcal{K}_G(\li_c).
\]
By Theorem \ref{maintheorem}, there exists $H_{Q,F}:\mathscr{C}(Q)\to\mathscr{C}(F)$,
$H_{F,G}:\mathscr{C}(F)\to\mathscr{C}(G)$ and $H_{G,Q}:\mathscr{C}(G)\to\mathscr{C}(Q)$.
On the other hand, they are not conjugate 
since the turning point of $q,f,g$ is respectively an attracting, 
neutral and repelling fixed point. 

\end{example}


\subsection{Generalized toy models}\label{ApenA}

In this section we introduce a larger class of toy models, changing unimodal to multimodal maps.
We say that an interval map $f:[0,1]\to[0,1]$ is an {\em $\ell$-modal map} if it has 
exactly $\ell$ turning points $c_1<\cdots<c_\ell$. This means that 
the intervals $[0,c_1],[c_1,c_2],\ldots,[c_\ell,1]$ are the largest intervals in which $f$ is monotone.
Fixed  $c_1<\cdots<c_\ell$, consider a one-parameter family
\[
f(y):[0,1]\to[0,1],\ y\in[0,1],
\]
of $\ell$-modal maps with turning points $c_1<\cdots<c_\ell$ such that $y\in[0,1]\mapsto f(y)$ is
continuous. We write $c_0=0,c_{\ell+1}=1$, and 
$I_1=[c_0,c_1),I_2=(c_1,c_2),\ldots,I_{\ell+1}=(c_\ell,c_{\ell+1}]$ for the intervals of monotonicity.
For each $y\in[0,1]$, let $f_i(y)$ be $f(y)\restriction_{I_i}$, see Figure \ref{f(yy)}.
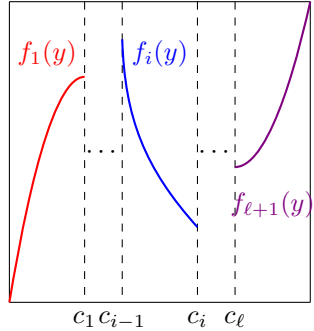
\begin{figure}[H]\centering
\begin{tikzpicture}[scale=1]
      \draw[] (3,-1) -- (-1,-1);
	  \draw[] (3,3) -- (3,-1);
	  \draw[] (3,3) -- (-1,3);
	  \draw[] (-1,3) -- (-1,-1);
	  \draw[dashed] (0,3) -- (0,-1) node[below]{$c_1$};
      \draw[dashed] (0.5,3) -- (0.5,-1) node[below] {$c_{i-1}$};
	  \draw[dashed] (1.5,3) -- (1.5,-1) node[below]{$c_i$};
      \draw[dashed] (2,3) -- (2,-1) node[below] {$c_\ell$};
      \draw[color=red,thick] (0,2) parabola (-1,-1) ;
	  \draw[color=violet,thick] (2,0.8) parabola (3,3) ;
	  \fill[black] (0.25,0.8) circle (0.mm) node[above] {$\cdots$};
	  \fill[black] (1.75,0.8) circle (0.mm) node[above] {$\cdots$};
	  \fill[red,thick] (-0.5,2) circle (0.mm) node[above] {$f_1(y)$};
	  \fill[blue,thick] (1,2) circle (0.mm) node[above] {$f_i(y)$};
	  \fill[violet,thick] (2.5,0) circle (0.mm) node[above] {$f_{\ell+1}(y)$};
      \coordinate (A) at (0.5,2.5);
      \coordinate (B) at (1.5,0);
	  \draw[color=blue,thick] (A) to [bend right=20] (B);
\end{tikzpicture}
\caption{Branches of $f(y)$.}
\label{f(yy)}
\end{figure}

\noindent
Let $0=a_1<b_1<a_2<b_2<\cdots<a_{\ell+1}<b_{\ell+1}=1$, and let
$k:[a_1,b_1]\cup\cdots\cup[a_{\ell+1},b_{\ell+1}]\to[0,1]$
be a differentiable map such that
$k(\{a_i,b_i\})\subset\{0,1\}$ and $\vert k'\vert>\gamma>1$. 
We define $K(x,y)$ by
\[
K(x,y)=K_i(y),\;\text{for}\; x\in I_i,
\]
where $K_i=(k\restriction_{[a_i, b_i]})^{-1}$, $i=1,\ldots, \ell+1$, are the inverse branches of $k$.

\medskip
\noindent
{\sc Generalized Toy Model:} It is the map defined by
\[
\begin{array}{crcl}
F \ : & \! ([0,1]\backslash\{c_1,\ldots,c_\ell\})\times [0,1]& \! \longrightarrow  
& \! [0,1]\times[0,1] \\
        & \! (x,y) & \! \longmapsto      & \! (f(y)(x), K(x,y)).
\end{array}
\]
We can similarly introduce extra points to make the domain of $F$ become compact.
For each $i=1,\ldots,\ell$, introduce points $c_i^i$ and $c_i^{i+1}$ and extend $F$ via the formula
$F(c^i_i,y)=(f_i(y)(c_i),K_i(y))$ and $F(c^{i+1}_i,y)=(f_{i+1}(y)(c_i),K_{i+1}(y))$, see Figure \ref{AcFitmg}.

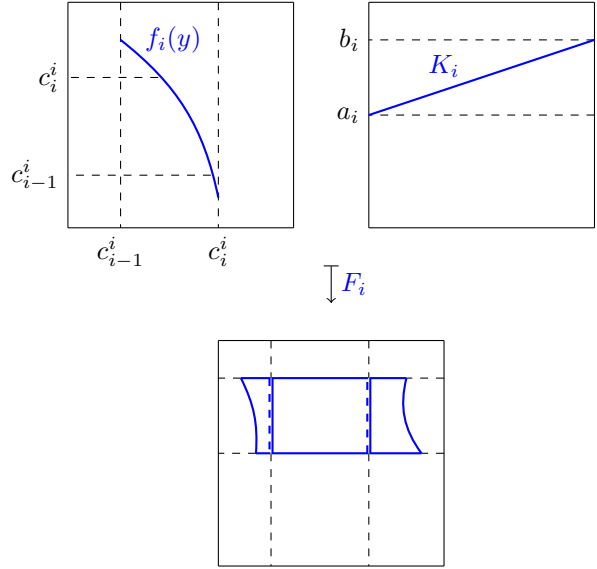
\begin{figure}[H]\centering
  \begin{tikzpicture}
	\draw[] (0,0) -- (3,0) ;
	\draw[] (0,0) -- (0,3) ;
	\draw[] (0,3) -- (3,3) ;
	\draw[] (3,0) -- (3,3) ;
	\draw[dashed] (0.7,3) -- (0.7,0) ;
	\draw[dashed] (1.9,0.7) -- (0,0.7) ;
	\draw[dashed] (2,3) -- (2,0) ;
	\draw[dashed] (1.2,2) -- (0,2) ;
	\fill[black] (0.7,0) circle (0.mm) node[below] {$c_{i-1}^i$};
	\fill[black] (2,0) circle (0.mm) node[below] {$c_i^i$};
	\fill[black] (0,0.7) circle (0.mm) node[left] {$c_{i-1}^i$};
	\fill[black] (0,2) circle (0.mm) node[left] {$c_i^i$};
	\fill[blue,thick] (1.4,2.2) circle (0.mm) node[above] {$f_i(y)$};
	\coordinate (l) at (0.7,2.5);
  \coordinate (m) at (2,0.4);
	\draw[color=blue,thick] (l) to [bend left=20] (m);
	\draw[] (4,0) -- (7,0) ;
	\draw[] (4,0) -- (4,3) ;
	\draw[] (4,3) -- (7,3) ;
	\draw[] (7,0) -- (7,3) ;
	\draw[dashed] (7,2.5) -- (4,2.5) ;
	\draw[dashed] (7,1.5) -- (4,1.5) ;
	\fill[black] (4,2.5) circle (0.mm) node[left] {$b_i$};
	\fill[black] (4,1.5) circle (0.mm) node[left] {$a_i$};
	\draw[blue,thick] (4,1.5) -- (7,2.5) ;
	\fill[blue,thick] (5,1.87) circle (0.mm) node[above] {$K_i$};
	\draw[|->] (3.5,-0.5) -- (3.5,-1);
	\fill[blue,thick] (3.5,-0.75) circle (0.mm) node[right] {$F_i$};
	\draw[] (2,-1.5) -- (5,-1.5) ;
	\draw[] (2,-1.5) -- (2,-4.5) ;
	\draw[] (2,-4.5) -- (5,-4.5) ;
	\draw[] (5,-4.5) -- (5,-1.5) ;
	\draw[dashed] (2.3,-2) -- (2,-2) ;
	\draw[dashed] (5,-2) -- (4.5,-2) ;
	\draw[dashed] (2.5,-3) -- (2,-3) ;
	\draw[dashed] (5,-3) -- (4.7,-3) ;
	\draw[dashed] (2.7,-3) -- (2.7,-4.5) ;
	\draw[dashed] (2.7,-1.5) -- (2.7,-2) ;	
	\draw[dashed] (4,-1.5) -- (4,-2) ;
	\draw[dashed] (4,-3) -- (4,-4.5) ;
	\draw[blue,thick] (2.5,-3) -- (2.69,-3) ;
	\draw[dashed,blue,thick] (2.68,-2) -- (2.68,-3) ;
	\draw[blue,thick] (2.71,-3) -- (3.99,-3) ;
	\draw[blue,thick] (2.72,-2) -- (2.72,-3) ;
	\draw[blue,thick] (4.01,-3) -- (4.7,-3) ;
	\draw[blue,thick] (2.3,-2) -- (2.69,-2) ;
	\draw[blue,thick] (2.71,-2) -- (3.99,-2) ;
	\draw[dashed,blue,thick] (3.98,-3) -- (3.98,-2) ;
	\draw[blue,thick] (4.02,-2) -- (4.02,-3) ;
	\draw[blue,thick] (4.01,-2) -- (4.5,-2) ;
	\coordinate (a) at (2.3,-2);
  \coordinate (b) at (2.5,-3);
	\draw[color=blue,thick] (a) to [bend left=15] (b);
	\coordinate (c) at (4.5,-2);
  \coordinate (d) at (4.7,-3);
	\draw[color=blue,thick] (c) to [bend right=25] (d);
	
\end{tikzpicture}
\caption{Action of $F_i$ on $[c_{i-1}^i,c_i^i]\times [0,1]$.}
\label{AcFitmg}
\end{figure}

For  $i=1,\ldots,\ell+1$, let $I_i$ be an interval on the construction of the family $y\mapsto f(y)$. 
We make the following convention.
If $f(y)$ is strictly increasing on $I_i$, we put
\[
\begin{array}{crcl}
F_i \ : & \! [c_{i-1}^i,c_i^i]\times [0,1]& \! \longrightarrow  
& \! ([0,c_1^1]\cup(c_1^2,c^2_2]\cup\cdots\cup(c_\ell^{\ell+1},1])\times[0,1] \\
        & \! (x,y) & \! \longmapsto      & \! (f_i(y)(x), K_i(y)).
\end{array}
\]
If $f(y)$ is strictly decreasing on $I_i$, we put
\[
\begin{array}{crcl}
F_i \ : & \! [c_{i-1}^i,c_i^i]\times [0,1]& \! \longrightarrow  
& \! ([0,c_1^1)\cup[c_1^2,c^2_2)\cup\cdots\cup[c_\ell^{\ell+1},1])\times[0,1] \\
        & \! (x,y) & \! \longmapsto      & \! (f_i(y)(x), K_i(y)).
\end{array}
\]
For each $(x,y)\in\text{Dom}(F)$, there exists
$j(x,y)=(j_1j_2\cdots j_m\cdots)\in\{1,\ldots, \ell+1\}^\N$
such that
\[
F^{m+1}(x,y)=F_{j_{m+1}}(x_m,y_m),\ \forall m\geq 0,
\]
where $y_m=(K_{j_m}\circ\cdots\circ K_{j_1})(y)=:y_{j_m\cdots j_1}$ and
$x_m=f_{j_m}(y_{j_{m-1}\cdots j_1})\circ\cdots\circ f_{j_2}(y_{j_1})\circ f_{j_1}(y)(x)$. 
Consider the alphabet $\mathcal{A}=\{I_1,c_1^1,c_1^2,I_2,\ldots,
c_\ell^\ell,c_\ell^{\ell+1},I_{\ell+1}\}$.

\medskip
\noindent
{\sc Address of a point:} The {\em address} of a point $(x,y)\in\text{Dom}(F)$ is 
the letter $t_F(x,y)\in\mathcal A$ defined by
$$
t_F(x,y)= \left\{
\begin{array}{ll}
I_i     & \textrm{, if }\pi_1(x,y)\in I_i     \\
c_i^i & \textrm{, if } \pi_1(x,y)=c^i_i \\
c_i^{i+1} & \textrm{, if } \pi_1(x,y)=c^{i+1}_i.
\end{array}
\right.
$$

\medskip
\noindent
{\sc Itinerary of a point:} The {\em itinerary} of a point $(x,y)\in\text{Dom}(F)$
is the sequence $T_F(x,y)\in\mathcal A^{\{0,1,2,\ldots\}}$ defined by 
\[
T_F(x,y)=(t_F(x,y),t_F(F(x,y)),\ldots,t_F(F^n(x,y)),\ldots).
\]

\medskip
\noindent
{\sc Critical Line:} The {\em critical line} of $F$ is the set
\[
\li_c(F)=\{(c_i^i, y),(c_i^{i+1},y): y\in[0,1]\;\text{and}\;i=1,\ldots, \ell\}.
\]

\medskip
\noindent
{\sc Kneading Sequences:} The {\em kneading sequences} of $F$ are all the itineraries
of elements of $\li_c(F)$. We denote this set of itineraries by $\mathcal{K}_F(\li_c)$. 

\medskip
With these more general definitions, a result similar to Theorem \ref{TheA} holds. 

\begin{theorem}\label{theorem-generalized}
Let $F,G$ be two generalized toy models, and assume that
neither of them has wandering intervals, interval of periodic points nor
weakly attracting periodic points. Then $F,G$ are fiber topologically conjugate 
if and only if they have the same kneading sequences.
\end{theorem}

The proof of Theorem \ref{theorem-generalized} follows exactly the same ideas and arguments
as in the proof of Theorem \ref{TheA}, but it demands a much heavier notation. Therefore we decided
not to include its proof and leave the details to the interested reader.


\section{Proofs of the main results}\label{section-proofs}

We now present the proofs of our main results. Given a set $X$, let $\partial X$ denote its boundary.

\subsection{Proof of Theorem \ref{maintheorem}}

A similar proof of Theorem \ref{maintheorem} can be found in \cite[Thm1]{Rand}.
Remind that
\[
\cc(F)=\left\{(x,y)\in\text{Dom}(F):\exists\; \ell\geq0 
\;\textrm{such that}\; F^\ell(x,y)\in\li_c(F)\right\}
\]
and that for $y\in[0,1]$ and $n\geq1$ we have
\[
\cc_n^F(y):=\left\{(x,y)\in\text{Dom}(F)\cap I(y):
\;F^\ell(x,y)\in\li_c(F) 
\;\textrm{for some}\; 0\leq \ell\leq n-1\right\}.
\]
Setting
\[
\cc^F(y):=\displaystyle\bigcup_{n\geq 1}\cc_n^F(y),
\]
\\
we have
\[
\cc(F)=\coprod_{y\in[0,1]} \cc^F(y).
\]
\\
We can similarly define the sets $\cc(G),\cc_n^G(y),\cc^G(y)$.
Suppose that, for all $y\in[0,1]$ and $n\geq1$, there exists a strictly 
increasing bijection $H_n(y):\cc_n^F(y)\to\cc_n^G(\psi(y))$ such that:
\begin{enumerate}[$\circ$]
\item $H_n(y)\circ F=G\circ H_n(y)$ on $\cc_n^F(y)\backslash\{(0^\pm,y)\}$,
\item $H_n(y)\restriction_{\cc_{n-1}^F(y)}=H_{n-1}(y)$.
\end{enumerate}
Since $\psi$ is a bijection, we have
\[
\cc(G)=\coprod_{y\in[0,1]}\cc^G(\psi(y)).
\]
Therefore, we can define $H:\cc(F)\longrightarrow \cc(G)$ by
$H(x,y)=H_n(y)(x,y)$ where $n$ is some (any) integer for which $(x,y)\in\cc_n^F(y)$.
It follows directly that $H\circ F=G\circ H$ on $\mathscr{C}(F)\backslash\li_c(F)$.
Hence the proof of Theorem \ref{maintheorem} is reduced to 
constructing the maps $H_n(y)$. This is what we will now do.

Let $\pp_{n}^F(y)$ be the partition of $I(y)$ induced by $\cc_n^F(y)$.
More specifically,
\[
\pp_{n}^F(y):=\left\{I_i^F(y)\subset I(y):\partial I_i^F(y)
\subset\cc_{n}^F(y)\cup\{(-1,y), (1,y)\},
1\leq i\leq k_n^F(y)\right\},
\]
where $k_n^F(y):=\#\pp_{n}^F(y)$ and the increasing order of the index of 
$I_i^F(y)$ is the same as the intervals are placed in $I(y)$. Note that the partition $\pp_{n}^F(y)$
includes two intervals, one with left endpoint $0^+$ and one with
right endpoint $0^-$.

Given a sequence
$j=j_1j_2\cdots j_m \in\{-, +\}^m$, we will use the following notation 
\[
f_j(y):=f_{j_m}(y_{j_{m-1}\cdots j_1})\circ\cdots\circ 
f_{j_3}(y_{j_2j_1})\circ f_{j_2}(y_{j_1})\circ f_{j_1}(y),
\]
where $y_{j_\ell\cdots j_1}=(K_{j_\ell}\circ\cdots\circ K_{j_1})(y)$
for $ 1\leq \ell\leq m$. 

Given $(x,y)\in\cc_n^F(y)\backslash\{(0^\pm,y)\}$, we will identify the pair $(x,y)$
with $x(y)_{j(x,y)}^F$, where $j(x,y)=j_1j_2\cdots j_k \in\{-, +\}^k$ is the minimal sequence such that
\[
F^k\left(x(y)_{j(x,y)}^F\right)=\left(f_{j(x,y)}(y)(x(y)_{j(x,y)}^F), 
y_{j_k\cdots j_1}\right)=\left(0^{j_k}, y_{j_k\cdots j_1}\right)
\]
for some $1\leq k\leq n-1$.
We will also abuse notation
by considering $x(y)_{j(x,y)}^F$ either as an ordered pair in order to apply $F$ or as a real number
in order to apply unimodal maps.

For each $I_i^F(y)\in\pp_{n}^F(y)$ there exists a unique minimal sequence
$j(I_i^F(y))=j^i_1j_2^i\cdots j_n^i\in\{-,+\}^n$
such that the function 
\[
(z,y)\in\textrm{int}\left(I_i^F(y)\right)\mapsto
\left(f_{j(I_i^F(y))}(y)(z),y_{j^i_n \cdots j^i_1}\right)
\] 
is strictly monotone. Hence
\[
\partial\left[I_i^F(y)\right]\subset\left\{(-1,y), (1,y), (0^\pm,y), 
x(y)_{j^i_1j_2^i\cdots j_k^i}^F\right\}
\]
for some $1\leq k\leq n-1$. For each $n\geq1$, set
\[
A_n^F(y):=\biggl\{j_1j_2\cdots j_k\in\{-, +\}^k:  
\begin{array}{c}
\exists (x,y)\in\cc_n^F(y)\backslash\{(0^\pm,y)\}\textrm{ such that }\\
(x,y)=x(y)^F_{j_1j_2\cdots j_k}\textrm{ for some }1\leq k\leq n-1
\end{array}
\biggr\}.
\]
We similarly define $\pp_{n}^G(\psi(y))$ and $A_n^G(\psi(y))$. 
Now, we will use induction on $n$ to complete the proof of the theorem.

Our underlying assumptions imply that we can define 
$H_1(y):\cc_1^F(y)\to\cc_1^G(\psi(y))$ with
$(0^\pm,y)\mapsto(0^\pm,\psi(y))$.
By induction, suppose that there exists a strictly increasing map
$H_n(y):\cc_n^F(y)\to\cc_n^G(\psi(y))$ such that
$H_n(y)\circ F=G\circ H_n(y)$ on $\cc_n^F(y)\backslash\{(0^\pm,y)\}$ for every $y\in[0,1]$. 
In particular, $\#\pp_{n}^F(y)=\#\pp_{n}^G(\psi(y))$ and 
$A_n^F(y)=A_n^G(\psi(y))$ for all $y\in[0,1]$.

Let $I_i^F(y)\in\pp_{n}^F(y)$. There is a unique sequence 
$j=j_1j_2\cdots j_n\in\{-,+\}^n$ with 
\[
\partial\left[I_i^F(y)\right]\subset
\left\{(\pm1, y),(0^\pm, y),x(y)_{j_1\cdots j_k}^F\right\}
\]
for some $1\leq k\leq n-1$. We analyse the possibilities for the interval $I_i^F(y)$.
The first possibility is
\[
I_i^F(y)=\left[\left[x(y)_{j_1\cdots j_k}^F, x(y)_{j_1\cdots j_\ell}^F\right],y\right]
\]
for $1\leq k\neq \ell\leq n-1$. By the induction hypothesis, we have
\[
I_i^G(\psi(y))=\left[\left[x(\psi(y))_{j_1\cdots j_k}^G, 
x(\psi(y))_{j_1\cdots j_\ell}^G\right],y\right].
\]
Moreover, $H_n(y)(x(y)_{j_1\cdots j_k}^F)=x(\psi(y))_{j_1\cdots j_k}^G$
and $H_n(y)(x(y)_{j_1\cdots j_\ell}^F)=x(\psi(y))_{j_1\cdots j_\ell}^G$.
Hence $F^n[I_i^F(y)]$ is equal to the union
$$
\begin{array}{c}
\bigl[\bigl(f_{j_n}(y_{j_{n-1}\cdots j_1})
\circ\cdots\circ f_{j_{k+1}}(y_{j_{k}\cdots j_1})(0^{j_{k+1}}),\hspace{1cm} \\ 
f_{j_n}(y_{j_{n-1}\cdots j_1})
\circ\cdots\circ f_{j_{\ell+1}}(y_{j_{\ell}\cdots j_1})(0^{j_{\ell+1}})\bigr)
, y_{j_n\cdots j_1}\bigr]
\end{array}
\cup\Lambda(I_i^F(y)) 
$$
where $\Lambda(I_i^F(y))=\{F^{n-k}(0^{j_k},y_{j_k\cdots j_1}),
F^{n-\ell}(0^{j_\ell},y_{j_\ell\cdots j_1})\}$, and $G^n[I_i^G(\psi(y))]$ is equal to the union
$$
\begin{array}{c}
\bigl[\bigl(g_{j_n}(\psi(y)_{j_{n-1}\cdots j_1})
\circ\cdots\circ g_{j_{k+1}}(\psi(y)_{j_{k}\cdots j_1})(0^{j_{k+1}}),\hspace{1cm} \\
g_{j_n}(\psi(y)_{j_{n-1}\cdots j_1})
\circ\cdots\circ g_{j_{l+1}}(\psi(y)_{j_{\ell}\cdots j_1})(0^{j_{\ell+1}})\bigr)
, y_{j_n\cdots j_1}\bigr] 
\end{array}
\cup\;\Lambda(I_i^G(\psi(y)))
$$
where $\Lambda(I_i^G(\psi(y)))=\{G^{n-k}(0^{j_k},\psi(y)_{j_k\cdots j_1}),
G^{n-\ell}(0^{j_\ell},\psi(y)_{j_\ell\cdots j_1})\}$.

Since $\mathcal{K}_F(\li_c)=\mathcal{K}_G(\li_c)$, we have
\[
F^n[I_i^F(y)]\cap\li_c(F)\neq\emptyset \ \iff \ G^n[I_i^G(\psi(y))]\cap\li_c(G)\neq\emptyset.
\]
If $F^n[I_i^F(y)]\cap\li_c(F)=\emptyset$, 
then $I_i^F(y)\in\pp_{n+1}^F(y)$.
On the other hand, if 
$\Lambda(I_i^F(y))\cap\li_c(F)\neq\emptyset$
we get $I_i^F(y)\in\pp_{n+1}^F(y)$ since
$F$ is injective. In any case, we also have
$I_i^G(\psi(y))\in\pp_{n+1}^G(\psi(y))$.

Now, if $\left(F^n[I_i^F(y)]\backslash\Lambda(I_i^F(y))\right)
\cap\li_c(F)\neq\emptyset$ then there is a unique point
\[
x(y)_{j_1j_2\cdots j_n}^F:=\left(f_{j_1}^{-1}(y)\circ\cdots\circ
f_{j_n}^{-1}(y_{j_{n-1}\cdots j_1})(0^{j_n}),y\right)\in
{\rm int}\left(I_i^F(y)\right),
\]
and so $\left(G^n[I_i^G(\psi(y))]\backslash\Lambda(I_i^G(\psi(y)))\right)
\cap\li_c(G)\neq\emptyset$, which in turn implies the existence of a unique point
\[
x(\psi(y))_{j_1j_2\cdots j_n}^G:=\left(g_{j_1}^{-1}(\psi(y))\circ\cdots\circ
g_{j_n}^{-1}(\psi(y)_{j_{n-1}\cdots j_1})(0^{j_n}),\psi(y)\right)\in
\textrm{int}\left(I_i^G(\psi(y))\right).
\]
Therefore we can define 
$$
H_{n+1}(y)(x(y)_{j_1j_2\cdots j_n}^F):=x(\psi(y))_{j_1j_2\cdots j_n}^G.
$$
The other possibilities $I_i^F=[[-1, x(y)^F_{j_1j_2\cdots j_k}], y]$,
$I_i^F=[[x(y)^F_{j_1j_2\cdots j_k}, 1], y]$,
and $I_i^F=[[x(y)^F_{j_1j_2\cdots j_k}, 0^\pm], y]$, $1\leq k\leq n-1$, are treated similarly.
Finally, defining $H_{n+1}(x,y)=H_n(x,y)$ for $(x,y)\in\cc_n^F(y)$, we have just constructed
a strictly increasing map
\[
H_{n+1}(y):\cc_{n+1}^F(y)\to\cc_{n+1}^G(\psi(y))
\]
such that:
\begin{enumerate}[$\circ$]
\item $H_{n+1}(y)\circ F=G\circ H_{n+1}(y)$ on $\cc_{n+1}^F(y)\backslash\{(0^\pm,y)\}$, and
\item $H_{n+1}(y)\restriction_{\cc_n^F(y)}=H_n(y)$. 
\end{enumerate}
This concludes the proof of Theorem \ref{maintheorem}.


\subsection{Proof of Theorem \ref{TheA}}

First, we shall prove some auxiliary lemmas.

\begin{lemma}\label{lemmathe}
Let $y\mapsto f(y)$ be the one-parameter family of unimodal maps, as in the definition of toy models.
For every $y\in[0,1]$ and $\varepsilon>0$
there is $\delta>0$ satisfying the following: if $y'\in[0,1]$, $x,x'\in[-1,1]$
are such that $d(y,y')<\delta$, $d(x,x')<\delta$ and there exist
$x(y)=f^{-1}_j(y)(x)$ and $x'(y')=f^{-1}_j(y')(x')$, where $j=-\; or\; +$,
then $d(x(y),x'(y'))<\varepsilon$.
\end{lemma}

\begin{proof}
The proof is by contradiction.
Suppose there is $\varepsilon>0$ such that for all $n\geq1$ there are
$y_n\in[0,1]$ and $x_n^y,x_n^{y_n}\in[-1,1]$ with $d(y,y_n)<\tfrac{1}{n}$ 
and $d(x_n^y,x_n^{y_n})<\tfrac{1}{n}$ such that $x_n(y)=f^{-1}_j(y)(x_n^y)$ and
$x_n(y_n)=f^{-1}_j(y_n)(x_n^{y_n})$ exist and $d(x_n(y),x_n(y_n))\geq\varepsilon$.
By compactness, we may suppose that $\lim x_n^y=\lim x_n^{y_n}=x_0$,
$\lim x_n(y)=x_1$ and $\lim x_n(y_n)=x_2$.
Note that $x_1\neq x_2$ and that either $x_1,x_2\in[-1,0]$ or $x_1,x_2\in[0,1]$. 
In particular $f(y)(x_1)\neq f(y)(x_2)$. But, using that $f(y_n)$ 
converges uniformly to $f(y)$, we have
$f(y)(x_1)=\lim f(y)(x_n(y))=\lim x_n^y=\lim x_n^{y_n}=\lim f(y_n)(x_n(y_n))=f(y)(x_2)$,
a contradiction.
\end{proof}

Let $\xi_-:([-1,0^-]\cup(0^+,1])\times[0,1]\to[-1,0^-]$ and 
$\xi_+:([-1,0^-)\cup[0^+,1])\times[0,1]\to[0^+,1]$
be two maps defined by 
$$
\xi_j(x,y)= \left\{
\begin{array}{cl}
f_j^{-1}(y)(x)  & \textrm{, if }\; x\in\text{Im}(f_j(y)) \\
0^j             & \text{, otherwise}.
\end{array}
\right.
$$
Lemma \ref{lemmathe} has the following consequence.

\begin{corollary}\label{corothem}
The maps $\xi_j$ defined above are continuous.
\end{corollary}

\begin{proof}
Let $(x,y)\in\text{Dom}(\xi_j)$ and $\varepsilon>0$. Assume that 
$x\in[-1, f_j(y)(0^j))$. Since the family $y\mapsto f(y)$
depends continuously on $y$, there exists $\delta>0$ such that
for every $(x',y')\in\text{Dom}(\xi_j)$ satisfying $d(x,x')<\delta$ and $d(y,y')<\delta$,
we have $x'\in[-1, f_j(y')(0^j))$.
From Lemma \ref{lemmathe}, if $\delta>0$ is small enough then
\[
d(\xi_j(x,y),\xi_j(x',y'))=d(f^{-1}_j(y)(x),f^{-1}_j(y')(x'))<\varepsilon.
\]
Suppose now that $x>f_j(y)(0^j)$. Again by continuity, there exists
$\delta>0$ such that for all $(x',y')\in\text{Dom}(\xi_j)$ satisfying $d(x,x')<\delta$ 
and $d(y,y')<\delta$ we have $x'>f_j(y')(0^j)$. Therefore,
$\xi_j(x,y)=0^j=\xi_j(x',y')$ and so 
\[
d(\xi_j(x,y),\xi_j(x',y'))=0.
\]
The case $x=f_j(y)(0^j)$ is treated analogously.

\end{proof}

The following lemma is fundamental to prove the 
continuity of the conjugacy in Theorem \ref{TheA}.

\begin{lemma}\label{lemmathe2}
For every $(x,y)\in\cc(F)$ there exists a continuous curve 
$\gamma^F:[0,1]\to\cc(F)$ of the form $\gamma^F(w)=(\widetilde{\gamma}^F(w),w)$ 
such that $\gamma^F(y)=(x,y)$.
\end{lemma}

\begin{proof}
Given $(x,y)\in\cc(F)$ there exists $n\geq1$ and a 
sequence $(j_1\cdots j_n)\in\{-,+\}^n$
such that $(x,y)=(f^{-1}_{j_1}(y)\circ\cdots\circ 
f^{-1}_{j_n}(y_{j_{n-1}\cdots j_1})(0^{j_n}),y)$.
Consider a sequence of curves defined inductively as follows:
\begin{enumerate}[$\circ$]
\item $\gamma^F_{j_1}:[0,1]\to\text{Dom}(F)$ is defined by 
$\gamma^F_{j_1}(w)=(\xi_{j_1}(0^{j_1},w),w)$.
\item $\gamma^F_{j_1j_2}:[0,1]\to\text{Dom}(F)$ is defined by 
$\gamma^F_{j_1j_2}(w)=(\xi_{j_1}(\pi_1(\gamma^F_{j_2}(w_{j_1})),w),w)$.
\item$\cdots$
\item $\gamma^F_{j_1\cdots j_n}:[0,1]\to\text{Dom}(F)$ is defined by 
$\gamma^F_{j_1\cdots j_n}(w)=(\widetilde{\gamma}^F_{j_1\cdots j_n}(w),w)$,
where $\widetilde{\gamma}^F_{j_1\cdots j_n}(w):=
\xi_{j_1}(\pi_1(\gamma^F_{j_2\cdots j_n}(w_{j_1})),w)$.
\end{enumerate}
From Corollary \ref{corothem}, each $\gamma^F_{j_1\cdots j_n}:
[0,1]\to\text{Dom}(F)$ is a continuous curve. 
By construction, $\gamma^F_{j_1\cdots j_n}(y)=(x,y)$.
Now it is quite easy to see that $\gamma^F_{j_1\cdots j_n}([0,1])\subset\cc(F)$.
For each $w\in[0,1]$, put 
\[
\ell_{j_1\cdots j_n}(w):=\max\{
1\leq k\leq n: f^{-1}_{j_1}(w)\circ\cdots\circ 
f^{-1}_{j_k}(w_{j_{k-1}\cdots j_1})(0^{j_k})\; \text{exists}\}.
\]
Hence, if $\ell_{j_1\cdots j_n}(w)$ exists then
\[
\gamma^F_{j_1\cdots j_n}(w)=\left(f^{-1}_{j_1}(w)\circ\cdots\circ 
f^{-1}_{j_{\ell_{j_1\cdots j_n}(w)}}(w_{j_{\ell_{j_1\cdots j_n}
(w)-1}\cdots j_1})(0^{j_{\ell_{j_1\cdots j_n}(w)}}),w\right).
\]
On the other hand, if $\ell_{j_1\cdots j_n}(w)$ does not exist, then
\[
\gamma^F_{j_1\cdots j_n}(w)=(0^{j_1},w).
\]
In any case we have $\gamma^F_{j_1\cdots j_n}([0,1])\subset\cc(F)$.
\end{proof}

Here is a direct consequence of Lemma \ref{lemmathe2}:
if $(x,y),(z,w)\in\cc(F)$ are such that $\exists n\geq1$ and
$(j_1\cdots j_n)\in\{-,+\}^n$
with $(x,y)=(f^{-1}_{j_1}(y)\circ\cdots\circ 
f^{-1}_{j_n}(y_{j_{n-1}\cdots j_1})(0^{j_n}),y)$ and 
$(z,w)=(f^{-1}_{j_1}(w)\circ\cdots\circ 
f^{-1}_{j_n}(w_{j_{n-1}\cdots j_1})(0^{j_n}),w)$,
then there is a continuous curve 
$\gamma^F:[0,1]\to\cc(F)$ such that $\gamma^F(y)=(x,y)$ and $\gamma^F(w)=(z,w)$.

\begin{proof}[Proof of Theorem \ref{TheA}.]
Because of Proposition \ref{pprop}, we only need to prove the reverse implication.
We will employ similar notation to the one used in the proof of Theorem \ref{maintheorem}.
For $y\in[0,1]$, let $H(y):\cc^F(y)\to\cc^G(\psi(y))$ be the map defined by
$H(y)(x,y):=H_n(y)(x,y)$,
where $n$ is some (any) integer such that $(x,y)\in\cc_n^F(y)$. We will extend 
$H(y)$ to $I(y)$. Take $(z,y)\in\overline{\cc^F(y)}\backslash\cc^F(y)$
and suppose that there exist $(z_n,y)^j\in\cc^F(y)$, $j=1,2$, such that $(z_n,y)^1\uparrow (z,y)$
and $(z_n,y)^2\downarrow (z,y)$. The cases where we have only 
$(z_n,y)\uparrow (z,y)$ or $(z_n,y)\downarrow (z,y)$ are similar. Since $G$ has no wandering
intervals, no intervals of periodic points and no weakly attracting
periodic point, $\cc^G(\psi(y))$ is dense on $I(\psi(y))$.
Thus the limits $\lim H(y)((z_n,y)^1)=:H^1(y)(z,y)$ and 
$\lim H(y)((z_n,y)^2)=:H^2(y)(z,y)$ exist and $H^1(y)(z,y)=H^2(y)(z,y)$. 
Therefore, for $(z,y)\in\overline{\cc^F(y)}\backslash\cc^F(y)$ we can define
\[
H(y)(z,y):=\lim H(y)(z_n,y),
\]
where $(z_n,y)\in\cc^F(y)$ is some (any) sequence such that $(z_n,y)\to(z,y)$.
Since $\cc^F(y)$ is dense on $I(y)$, we obtain the map $H:\text{Dom}(F)\to\text{Dom}(G)$ 
defined by $H(x,y):=H(y)(x,y)$. Note that $H$ is strictly increasing on each fiber $I(y)$.
\newline

\medskip
\noindent
{\sc Claim:} The map $H$ is continuous. 
\newline

\begin{proof}[Proof of the claim.]
Let $(x,y)\in\text{Dom}(F)\backslash\{(\pm1,y),(0^\pm,y): y\in[0,1]\}$ (the case
$(x,y)\in\{(\pm1,y),(0^\pm,y): y\in[0,1]\}$ can be treated similarly). 
Let $B_\varepsilon(H(x,y))$ be a neighborhood of $H(x,y)$ for some $\varepsilon>0$.
Since $\overline{\cc^G(\psi(y))}=I(\psi(y))$,
there are $z_1<x<z_2$ such that $(z_1,y),(z_2,y)\in\cc(F)$ and
$H(z_i,y)\in B_\varepsilon(H(x,y))\cap I(\psi(y))$, $i=1,2$. Note that 
$\pi_1[H(z_1,y)]<\pi_1[H(x,y)]<\pi_1[H(z_2,y)]$. 
By Lemma \ref{lemmathe2},  there are 
continuous curves $\gamma^G_{j_1\cdots j_\ell}$ and $\gamma^G_{e_1\cdots e_k}$ 
such that $\gamma^G_{j_1\cdots j_\ell}(\psi(y))=H(z_1,y)$ and 
$\gamma^G_{e_1\cdots e_k}(\psi(y))=H(z_2,y)$. By continuity, there are
$y_1<y<y_2$ so that the domain $D(H(x,y))$ whose boundary is formed by the two curves 
$\gamma^G_{j_1\cdots j_\ell}([\psi(y_1),\psi(y_2)])$,
$\gamma^G_{e_1\cdots e_k}([\psi(y_1),\psi(y_2)])$ and the two intervals
$[(\gamma^G_{j_1\cdots j_\ell}(\psi(y_1)),\gamma^G_{e_1\cdots e_k}(\psi(y_1))), \psi(y_1)]$, 
$[(\gamma^G_{j_1\cdots j_\ell}(\psi(y_2)),\gamma^G_{e_1\cdots e_k}(\psi(y_2))), \psi(y_2)]$ 
contains $H(x,y)$ and satisfies $D(H(x,y))\subset B_\varepsilon(H(x,y))$, see Figure \ref{HHH}.
Again by Lemma \ref{lemmathe2}, for every sequence 
$t_1\cdots t_m\in\{-,+\}^m$ and $w\in[0,1]$ we get 
$H(\gamma^F_{t_1\cdots t_m}(w))=\gamma^G_{t_1\cdots t_m}(\psi(w))$.
Let $D(x,y)$ be the domain whose boundary is formed by the curves
$\gamma^F_{j_1\cdots j_\ell}([y_1,y_2])$,
$\gamma^F_{e_1\cdots e_k}([y_1,y_2])$ and the  intervals
$[(\gamma^F_{j_1\cdots j_\ell}(y_1),\gamma^F_{e_1\cdots e_k}(y_1)), y_1]$, 
$[(\gamma^F_{j_1\cdots j_\ell}(y_2),\gamma^F_{e_1\cdots e_k}(y_2)), y_2]$. 
Such domain contains a neighborhood of $(x,y)$. 
Since $H$ is monotone on fibers, we have $H(D(x,y))=D(H(x,y))$. 
This proves the claim.
\end{proof}

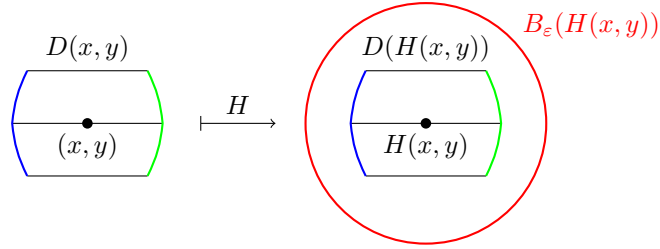
\begin{figure}[H]\centering
	\begin{tikzpicture}
\draw[color=black] (0,0) -- (2,0) ;
\draw[color=black] (0.2,-0.7) -- (1.8,-0.7) ;
\draw[color=black] (0.2,0.7) -- (1.8,0.7) ;
	\coordinate (a) at (0,0);
  \coordinate (b) at (0.2,0.7);
	\coordinate (c) at (0.2,-0.7);
	\draw[color=blue,thick] (c) to [bend left=10] (a) to [bend left=10] (b);
	\coordinate (d) at (2,0);
  \coordinate (e) at (1.8,0.7);
	\coordinate (f) at (1.8,-0.7);
	\draw[color=green,thick] (f) to [bend right=10] (d) to [bend right=10] (e);
\fill[black] (1,0) circle (0.mm) node[below]{$(x,y)$};
\fill[black] (1,0.7) circle (0.mm) node[above]{$D(x,y)$};
\fill[black] (1,0) circle (0.7 mm);
 \draw[|->] (2.5,0) -- (3.5,0) ;
\fill[black] (3,0) circle (0.mm) node[above]{$H$};
\fill[black] (5.5,0.7) circle (0.mm) node[above]{$D(H(x,y))$};
\fill[red] (7.7,1) circle (0.mm) node[above]{$B_\varepsilon(H(x,y))$};
\draw[color=black] (4.5,0) -- (6.5,0) ;
\draw[color=black] (4.7,-0.7) -- (6.3,-0.7) ;
\draw[color=black] (4.7,0.7) -- (6.3,0.7) ;
	\coordinate (A) at (4.5,0);
  \coordinate (B) at (4.7,-0.7);
	\coordinate (C) at (4.7,0.7);
	\draw[color=blue,thick] (C) to [bend right=10] (A) to [bend right=10] (B);
	\coordinate (D) at (6.5,0);
  \coordinate (E) at (6.3,0.7);
	\coordinate (F) at (6.3,-0.7);
	\draw[color=green,thick] (F) to [bend right=10] (D) to [bend right=10] (E);
\fill[black] (5.5,0) circle (0.mm) node[below]{$H(x,y)$};
\draw[red,thick] (5.5,0) circle (1.6cm);
\fill[black] (5.5,0) circle (0.7 mm);

\end{tikzpicture}
\caption{Action of $H$ on $D(x,y)$.}
\label{HHH}
\end{figure}

At the moment, we know that $H$ is injective (by construction) and surjective (by the claim).
It remains to show that $H\circ F=G\circ H$. From Theorem \ref{maintheorem},
we already have $H\circ F=G\circ H$ on $\cc(F)\backslash\li_c(F)$.
Now, take $(z,y)\in\text{Dom}(F)\backslash\cc(F)$
and let $(z_n,y)\in\cc^F(y)\backslash\{0^\pm,y\}$ such that $\lim(z_n,y)=(z,y)$
(the case $(z,y)\in\li_c(F)$ is treated analogously).
Using that $F$ is continuous on $\text{Dom}(F)\backslash\cc(F)$ and that
$H(z,y)\in\overline{\cc^G(\psi(y))}\backslash\cc^G(\psi(y))$, we conclude that
\[
(H\circ F)(z,y)=\lim (H\circ F)(z_n,y_n) = \lim (G\circ H)(z_n,y_n)=(G\circ H)(z,y).
\]
This finishes the proof of the theorem.
\end{proof}

We finish this section remarking that the same proof above can be used
to prove a variant of Theorem \ref{TheA}, with assumptions only on $G$: 
Let $F,G$ be two toy models, and assume that $G$ has no wandering intervals,
no interval of periodic points nor weakly attracting periodic points.
If $F,G$ have the same kneading sequences, then they are fiber topologically semiconjugate.


\subsection{Proof of Theorem \ref{maintheorem2}}
We begin with a lemma whose proof is similar to that of Proposition \ref{nonwanprop}.

\begin{lemma}\label{lthmB}
Let $F$ be a toy model. If $F$ has no wandering domain, then 
\[
\overline{\cc(F)}={\rm Dom}(F).
\]
\end{lemma}

\begin{proof}
By contradiction, assume that there exists an open set $V\subset\text{Dom}(F)$ such that 
$F^m(V)\cap\li_c(F)=\emptyset$ for all $m\geq0$.
Construct, as in the proof of Proposition \ref{nonwanprop}, 
a connected set $L$ containing $V$ such that $F^\ell(L)\subset L$ for some $0<\ell\leq m$,
$F^i(L)\cap\li_c(F)=\emptyset$ and $F^i(L)\cap F^j(L)=\emptyset$
for all $i,j=0,\ldots,\ell-1$ with $i\neq j$.
Therefore, there exists a sequence $j_0j_1\cdots j_{\ell-1}\in\{-,+\}^\ell$
such that, for any $(a,b)\in L$ and $n\geq0$ 
it holds $F^{n\ell}(a,b)=(a_{n\ell},b_{(j_{\ell-1}\cdots j_1j_0)^n})$, where 
$b_{(j_{\ell-1}\cdots j_1j_0)^n}:=\left(K_{j_{\ell-1}}\circ\cdots\circ K_{j_0}\right)^n(b)$  
and $a_{n\ell}=f_{j_{\ell-1}}(b_{j_{\ell-2}\cdots j_0(j_{\ell-1}\cdots j_0)^{n-1}})
\circ\cdots\circ f_{j_0}(b_{(j_{\ell-1}\cdots j_0)^{n-1}})\circ\cdots\circ
f_{j_{\ell-1}}(b_{j_{\ell-2}\cdots j_0})\circ\cdots\circ 
f_{j_1}(b_{j_0})\circ f_{j_0}(b)(a)$. Since $K_{j_{\ell-1}}\circ\cdots\circ K_{j_0}$ 
is a contraction, there is $w\in\pi_2(\overline{L})$ such that
$(K_{j_{\ell-1}}\circ\cdots\circ K_{j_0})(w)=w$. For each $(x,y)\in V$, we can 
prove that $y=w$. Therefore $V\subset I(w)$, 
which contradicts the fact that $V$ is an open set. 
\end{proof}

For $n\geq1$, let
\[
\cc_n^F=\left\{(x,y)\in\text{Dom}(F): F^\ell(x,y)\in\li_c(F)\text{ for some }0\leq \ell\leq n-1\right\}.
\]

For each $y\in[0,1]$, let $H_n(y):\cc_n^F(y)\to\cc_n^G(\psi(y))$ 
be the map constructed in the proof of Theorem \ref{maintheorem}. 
Let $\widetilde{H_n}(y):I(y)\to I(\psi(y))$ be the homeomorphism such that:
\begin{enumerate}[$\circ$]
\item $\widetilde{H_n}(y)(\pm1,y)=(\pm1,\psi(y))$,
\item $\widetilde{H_n}(y)\restriction_{\cc_n^F(y)}=H_n(y)$, and
\item $\widetilde{H_n}(y)$ is linear in each connected component of $I(y)\backslash\cc_n^F(y)$.
\end{enumerate}
Now define $\widetilde{H_n}:\text{Dom}(F)\to\text{Dom}(G)$ by
$\widetilde{H_n}(x,y):=\widetilde{H_n}(y)(x,y)$. 

\begin{lemma}
For each $n\geq1$, $\widetilde{H_n}$ is continuous.  
\end{lemma}

\begin{proof}
We start invoking the curves constructed in the proof of Lemma \ref{lemmathe2}: 
for each $j_1j_2\cdots j_\ell\in\{-,+\}^\ell$ with $0\leq \ell\leq n-1$, let
$\gamma^F_{j_1j_2\cdots j_\ell}:[0,1]\to\cc^F_n$
be the continuous curve appearing in the construction of $\cc_n^F$.
Note that $\cc_n^F$ is the union of finitely many such curves:
\begin{equation*}
\cc^F_n = 
\displaystyle\bigcup_{\substack{j_1j_2\cdots j_\ell\: \in\: \{-,+\}^\ell\\
                  0\:\leq\: \ell\:\leq\: n-1}}
        \gamma^F_{j_1j_2\cdots j_\ell}([0,1]). 
\end{equation*}
Take $(x,y)\in\text{Dom}(F)\backslash\cc_n^F$, and take 
$\delta>0$ such that  $B_\delta(x,y)\cap\cc_n^F=\emptyset$.
Let $x_1<x<x_2$ such that $(x_1,y),(x_2,y)\in\cc_n^F$ and
$\cc_n^F(y)\cap[(x_1,x_2),y]=\emptyset$ (if $(x_1,y)$ or $(x_2,y)$ does not
exist, we take $x_1=-1$ or $x_2=1$ accordingly).
Take curves $\gamma^F_{j_1j_2\cdots j_\ell}$ and 
$\gamma^F_{e_1e_2\cdots e_k}$ such that 
$\gamma^F_{j_1j_2\cdots j_\ell}(y)=(x_1,y)$ and  
$\gamma^F_{e_1e_2\cdots e_k}(y)=(x_2,y)$, for some $1\leq \ell,k\leq n-1$.
For any curve $\gamma^F_{t_1t_2\cdots t_i}$ with $1\leq i\leq n-1$, we have
$\gamma^F_{t_1t_2\cdots t_i}(y)\cap\left[(x_1,x_2),y\right]=\emptyset$.
See Figure \ref{Htil3}.

\begin{figure}[H]\centering
	\begin{tikzpicture}
\draw[color=black] (0,0) -- (3,0) ;
	\coordinate (a) at (0,0);
  \coordinate (b) at (-0.2,1);
	\coordinate (c) at (-0.2,-1);
	\draw[color=violet,thick] (c) to [bend right=10] (a) to [bend right=10] (b);
	\coordinate (A) at (0,0);
  \coordinate (B) at (-0.5,1.2);
	\coordinate (C) at (0.2,-1.3);
	\draw[color=blue,thick] (C) to [bend right=10] (A) to [bend right=10] (B);
	\coordinate (d) at (3,0);
  \coordinate (f) at (3.5,1.5);
	\coordinate (g) at (2.4,-1.2);
	\draw[color=red,thick] (g) to [bend right=10] (d) to [bend right=10] (f);
\fill[black] (1.5,0) circle (0.mm) node[below]{$(x,y)$};
\fill[black] (3,0) circle (0.mm) node[right]{$(x_2,y)$};
\fill[black] (0,0) circle (0.mm) node[left]{$(x_1,y)$};
\fill[black] (1.5,1.1) circle (0.mm) node[above]{\textcolor{violet}{$\gamma^F_{j_1j_2\cdots j_\ell}$}};
\fill[black] (1.5,1.7) circle (0.mm) node[above]{\textcolor{blue}{$\gamma^F_{t_1t_2\cdots t_i}$}};
\fill[black] (1.5,0.5) circle (0.mm) node[above]{\textcolor{red}{$\gamma^F_{e_1e_2\cdots e_k}$}};
\fill[black] (3,0) circle (0.7 mm);
\fill[black] (1.5,0) circle (0.7 mm);
\fill[black] (0,0) circle (0.7 mm);
\end{tikzpicture}
\caption{Behaviour near the interval $[(x_1,x_2),y]$ .}
\label{Htil3}
\end{figure}
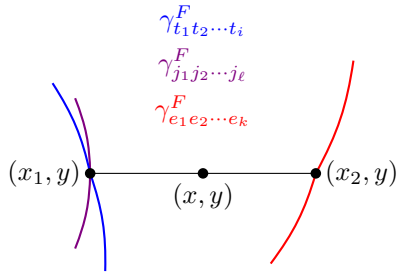
Choose $y_1<y<y_2$ and two continuous curves
$\alpha_1,\alpha_2:[y_1,y_2]\to\cc_n^F$ such that:
\begin{enumerate}[$\circ$]
\item $\alpha_i(y)=(x_i,y)$, $i=1,2$.
\item If $D(x,y)$ is the domain whose boundary is formed by the two curves $\alpha_1([y_1,y_2])$,
$\alpha_2([y_1,y_2])$ and the two intervals
$[(\alpha_1(y_1),\alpha_2(y_1)),y_1]$, $[(\alpha_1(y_2),\alpha_2(y_2)),y_2]$,
then $D(x,y)\cap\cc_n^F=\alpha_1([y_1,y_2])\cup
\alpha_2([y_1,y_2])$. 
\end{enumerate}
Then $\widetilde{H_n}(D(x,y))=:D(\widetilde{H_n}(x,y))$ is a domain such that
\[
D(\widetilde{H_n}(x,y))\cap\cc_n^G=\widetilde{\alpha_1}
[\psi(y_1),\psi(y_2)]\cup\widetilde{\alpha_2}[\psi(y_1),\psi(y_2)],
\]
where $\widetilde{\alpha_i}=\widetilde{H_n}(\alpha_i)$, see Figure \ref{Htil2}. 

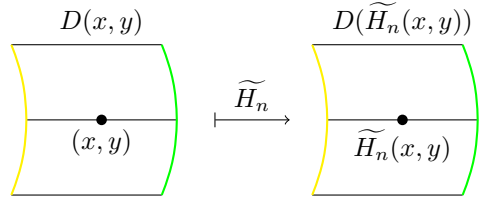
\begin{figure}[H]\centering
	\begin{tikzpicture}
\draw[color=black] (0,0) -- (2,0) ;
\draw[color=black] (-0.2,-1) -- (1.8,-1) ;
\draw[color=black] (-0.2,1) -- (1.8,1) ;
	\coordinate (a) at (0,0);
  \coordinate (b) at (-0.2,1);
	\coordinate (c) at (-0.2,-1);
	\draw[color=yellow,thick] (c) to [bend right=10] (a) to [bend right=10] (b);
	\coordinate (d) at (2,0);
  \coordinate (e) at (1.8,1);
	\coordinate (f) at (1.8,-1);
	\draw[color=green,thick] (f) to [bend right=10] (d) to [bend right=10] (e);
\fill[black] (1,0) circle (0.mm) node[below]{$(x,y)$};
\fill[black] (1,1) circle (0.mm) node[above]{$D(x,y)$};
\fill[black] (1,0) circle (0.7 mm);
\draw[|->] (2.5,0) -- (3.5,0) ;
\fill[black] (3,0) circle (0.mm) node[above]{$\widetilde{H_n}$};
\fill[black] (5,1) circle (0.mm) node[above]{$D(\widetilde{H_n}(x,y))$};
\draw[color=black] (4,0) -- (6,0) ;
\draw[color=black] (3.8,-1) -- (5.8,-1) ;
\draw[color=black] (3.8,1) -- (5.8,1) ;
	\coordinate (A) at (4,0);
  \coordinate (B) at (3.8,-1);
	\coordinate (C) at (3.8,1);
	\draw[color=yellow,thick] (C) to [bend left=10] (A) to [bend left=10] (B);
	\coordinate (D) at (6,0);
  \coordinate (E) at (5.8,1);
	\coordinate (F) at (5.8,-1);
	\draw[color=green,thick] (F) to [bend right=10] (D) to [bend right=10] (E);
\fill[black] (5,0) circle (0.mm) node[below]{$\widetilde{H_n}(x,y)$};
\fill[black] (5,0) circle (0.7 mm);
\end{tikzpicture}
\caption{Action of $\widetilde{H_n}$ on a neighborhood of $(x,y)$.}
\label{Htil2}
\end{figure}

By the definition of $\widetilde{H_n}$, for $(z,w)\in D(x,y)$ we have
$\widetilde{H_n}(z,w)=(\xi_{x,y}(z,w),\psi(w))$, where
\[
\xi_{x,y}(z,w)=\pi_1[\widetilde{\alpha_1}(\psi(w))]+
\dfrac{\pi_1[\widetilde{\alpha_2}(\psi(w))]-
\pi_1[\widetilde{\alpha_1}(\psi(w))]}{\pi_1[\alpha_2(w)]-\pi_1[\alpha_1(w)]} \left(z-\pi_1[\alpha_1(w)]\right).
\]
In particular, $\widetilde{H_n}$ is continuous at $(x,y)$.
Indeed, $\widetilde{H_n}$ is continuous on the union $D(x,y)\cup\partial D(x,y)$.

Suppose now that $(x,y)\in\cc_n^F$. Let $s\in\{-,+\}^\ell$
such that $(x,y)=(\gamma^F_s(y),y)$ for some $0\leq \ell\leq n-1$,
and consider $(x_j,y_j)\to(x,y)$. For each $j\geq 1$, there are curves 
$\gamma^F_{a_j},\gamma^F_{b_j}$ such that 
$(x_j,y_j)$ belongs to the interval $[[\gamma^F_{a_j}(y_j),\gamma^F_{b_j}(y_j)], y_j]$
and $[(\gamma^F_{a_j}(y_j),\gamma^F_{b_j}(y_j)), y_j]\cap\cc_n^F(y)=\emptyset$.
To prove that  
$\widetilde{H_n}(x_j,y_j)\to\widetilde{H_n}(x,y)$,
it is enough to prove that
any subsequence $(x_{j_k},y_{j_k})$
contains a sub-subsequence $(x_{j_{k_\ell}},y_{j_{k_\ell}})$ such that
$\widetilde{H_n}(x_{j_{k_\ell}},y_{j_{k_\ell}})\to\widetilde{H_n}(x,y)$.
Let $(x_{j_k},y_{j_k})$ be a subsequence of $(x_j,y_j)$.
There is a sub-subsequence with
$(x_{j_{k_\ell}},y_{j_{k_\ell}})\in[[\gamma^F_{a}(y_{j_{k_\ell}}),\gamma^F_{b}(y_{j_{k_\ell}})], 
y_{j_{k_\ell}}]$ for fixed $a,b$. Since 
$\gamma^F_{a}(y_{j_{k_\ell}})\leq x_{j_{k_\ell}}\leq\gamma^F_{b}(y_{j_{k_\ell}})$ for every $\ell$, 
when passing to the limit $\ell\to\infty$ we obtain $\gamma^F_{a}(y)\leq\gamma^F_s(y)\leq\gamma^F_{b}(y)$.
If $\gamma^F_{a}(y)=\gamma^F_{b}(y)$
then $\widetilde{H_n}(x_{j_{k_\ell}},y_{j_{k_\ell}})\to\widetilde{H_n}(x,y)$, since
\[
\widetilde{H_n}(\gamma^F_a(y_{j_{k_\ell}}),y_{j_{k_\ell}})\leq\widetilde{H_n}(x_{j_{k_\ell}},y_{j_{k_\ell}})
\leq\widetilde{H_n}(\gamma^F_b(y_{j_{k_\ell}}),y_{j_{k_\ell}})
\]
and 
\begin{align*}
&\lim_{\ell\to\infty}\widetilde{H_n}(\gamma^F_a(y_{j_{k_\ell}}),y_{j_{k_\ell}})=\widetilde{H_n}(\gamma^F_a(y),y)
=\widetilde{H_n}(\gamma^F_s(y),y)\\
&=\widetilde{H_n}(\gamma^F_b(y),y)=\lim_{l\to\infty}\widetilde{H_n}(\gamma^F_b(y_{j_{k_\ell}}),y_{j_{k_\ell}}).
\end{align*}
Assume now that $\gamma^F_{a}(y)\neq\gamma^F_{b}(y)$. Then either 
$\gamma^F_{s}(y)=\gamma^F_{a}(y)$ or $\gamma^F_{s}(y)=\gamma^F_{b}(y)$. 
In either case we can build a domain $D(x,y)$ such that $(x_{j_{k_\ell}},y_{j_{k_\ell}})\in D(x,y)$ 
for all $\ell\geq1$. Hence $\widetilde{H_n}(x_{j_{k_\ell}},y_{j_{k_\ell}})\to\widetilde{H_n}(x,y)$. 
It follows that $\widetilde{H_n}$ is continuous.
\end{proof}

Our next goal is to prove that $\{\widetilde{H_n}\}_{n\geq1}$ is a Cauchy sequence.

\begin{lemma}\label{SeqCau}
The sequence $\{\widetilde{H_n}\}_{n\geq1}$ is Cauchy.  
\end{lemma}

\begin{proof}
Let $\varepsilon>0$. By the equicontinuity of the family $\{\phi_n^G\}$, there is
$\delta>0$ such that if $|y_1-y_2|<\delta$ then 
$d_\mathcal{H}(\cc_n^G(y_1),\cc_n^G(y_2))<\varepsilon$ for all $n\geq1$.
By Lemma \ref{lthmB},
$\cc(G)$ is dense in $\text{Dom}(G)$, so there is $N>0$ large enough
such that $\{B_\delta(x,y);\; (x,y)\in\cc_N^G\}$ covers $\text{Dom}(G)$. 
Let $(x,y)\in\text{Dom}(F)$ and $m,n\geq N$. 
Since $\widetilde{H_n}$ is surjective, we can take $x_1<x<x_2$ such that
$|\widetilde{H_n}(x_1,y)-\widetilde{H_n}(x,y)|=
|\widetilde{H_n}(x_2,y)-\widetilde{H_n}(x,y)|=2\varepsilon$. Take
$(x_1',y'),(x_2',y'')\in\cc_N^F$ such that 
$\widetilde{H_n}(x_1,y)\in B_\delta(\widetilde{H_N}(x_1',y'))$
and $\widetilde{H_n}(x_2,y)\in B_\delta(\widetilde{H_N}(x_2',y''))$. 
Again by equicontinuity, there are
$(\widetilde{x_1},y),(\widetilde{x_2},y)\in\cc^F_N$ such that
$\widetilde{H_n}(x,y)$ belongs to the interval $[[\widetilde{H_N}(\widetilde{x_1},y),
\widetilde{H_N}(\widetilde{x_2},y)], \psi(y)]$
and $|\widetilde{H_N}(\widetilde{x_1},y) - 
\widetilde{H_N}(\widetilde{x_2},y)|<6\varepsilon$. Since 
$\widetilde{H_\ell}\restriction_{\cc_N^F}=\widetilde{H_N}\restriction_{\cc_N^F}$ 
$\forall\ell\geq N$,
we obtain that $\widetilde{H_n}(x,y),\widetilde{H_m}(x,y)$ both belong to the interval
$[[\widetilde{H_N}(\widetilde{x_1},y),\widetilde{H_N}(\widetilde{x_2},y)], \psi(y)]$. Hence
$$|\widetilde{H_n}(x,y) - \widetilde{H_m}(x,y)|<6\varepsilon,$$ and so
\[
\|H_n - H_m\|=\sup_{(x,y)\in\text{Dom}(F)}
|\widetilde{H_n}(x,y) - \widetilde{H_m}(x,y)|<6\varepsilon.
\]
This proves that $\{\widetilde{H_n}\}_{n\geq1}$ is a Cauchy sequence.
\end{proof}

Now we finish the proof of Theorem \ref{maintheorem2}. By the previous lemma,
$\widetilde{H}:\text{Dom}(F)\to\text{Dom}(G)$ defined by $\widetilde{H}:=\lim\widetilde{H_n}$ is 
continuous and surjective, so it remains to prove that $\widetilde{H}$ semiconjugates $F$ and $G$.

Since $\widetilde{H}$ agrees with $\widetilde{H_n}$ on 
$\cc^F_n$, we have $\widetilde{H}\circ F=G\circ\widetilde{H}$ on $\cc(F)\backslash\li_c(F)$.
Take $(x,y)\in\text{Dom}(F)\backslash\cc(F)$ and $(x_n,y_n)\in\cc(F)$
such that $\lim(x_n,y_n)=(x,y)$. 
Since $F$ is continuous on $\text{Dom}(F)\backslash\cc(F)$ and 
$\widetilde{H}(x,y)\in\text{Dom}(G)\backslash\cc(G)$,
\[
(\widetilde{H}\circ F)(x,y)=\lim(\widetilde{H}\circ F)(x_n,y_n)=
\lim (G\circ\widetilde{H})(x_n,y_n)=(G\circ\widetilde{H})(x,y).
\]
Take now $(x,y)\in\li_c(F)$, say $(x,y)=(0^-,y)$ (the case $(x,y)=(0^+,y)$ is similar). 
If $F(0^-,y)\notin\li_c(F)$, then for any sequence 
$(x_n,y_n)\in\cc(F)\backslash\li_c(F)$ with $\lim(x_n,y_n)=(0^-,y)$ we have $\lim F(x_n,y_n)=F(0^-,y)$. Hence
$(\widetilde{H}\circ F)(0^-,y)=(G\circ\widetilde{H})(0^-,y)$. Now assume that $F(0^-,y)\in\li_c(F)$, then
\[
(\widetilde{H}\circ F)(0^-,y)=(0^-,\psi\circ K_-^F(y))=(0^-, K_-^G\circ\psi(y))
=G(0^-,\psi(y))=(G\circ\widetilde{H})(0^-,y),
\]
where in the third passage we used that the kneading sequences 
$\mathcal{K}_F(\li_c)$ and $\mathcal{K}_G(\li_c)$ are the same.
Therefore $\widetilde{H}\circ F=G\circ\widetilde{H}$ on $\text{Dom}(F)$,
which concludes the proof.


\subsection{Proof of Theorem \ref{thesinger222}}

Let $f:I\to I$ be a $C^3$ interval map and let $x\in I$ 
such that $Df(x)\neq0$.

\medskip
\noindent
{\sc Schwarzian Derivative:} The {\em Schwarzian derivative}
$Sf(x)$ of $f$ at $x$ is
\[
Sf(x):= \tfrac{D^3f(x)}{Df(x)}-\tfrac{3}{2}\left(\tfrac{D^2f(x)}{Df(x)}\right)^2.
\]

We say that $f$ has {\it negative Schwarzian derivative} if $Sf(x)<0$
for all $x\in I$ except, possibly, the turning points.
At these points we define $Sf(x)=-\infty$.
It is easily seen from the definition that if $f,g:I\to I$ are $C^3$ and
$x\in I$ satisfies $Df(x)\neq0$ and $Dg(f(x))\neq0$ then
\[
S(g\circ f)(x)=Sg(f(x))\cdot(Df(x))^2 + Sf(x).
\] 
This implies the following: if $f,g$ have negative Schwarzian derivative,
then $g\circ f$ also has negative Schwarzian derivative.

\begin{lemma}[Minimum Principle]\label{minpri}
Let $I=[a,b]$ and $f:I\to I$ a $C^3$ map with negative Schwarzian derivative. 
If $Df(x)\neq0$ for all $x\in I$, then
\[
|Df(x)|>\min\{|Df(a)|, |Df(b)|\}, \ \forall x\in(a,b).
\]
\end{lemma}

\begin{proof}
See \cite[Lemma 6.1]{dMvS}.
\end{proof}

We call $c\in I$ a {\em critical point} of $f$ if $Df(c)=0$. The following theorem, due to Singer,
is a consequence of the Minimum Principle.

\begin{theoremm}[Singer]
If $f:I\to I$ is a $C^3$ map with negative Schwarzian
derivative then the immediate basin of any 
attracting periodic orbit contains either a critical 
point of $f$ or a boundary point of $I$.
\end{theoremm}

\begin{proof}
See \cite[Thm 6.1]{dMvS}.
\end{proof}

Let us restate Theorem \ref{thesinger222},
which is a version of Singer's theorem for the toy models.

\begin{teoBBB}
Let $F(x,y)=(f(y)(x), K^F_{{\rm sign}(x)}(y))$ be a toy model.
If each interval map $f(y)$ has negative Schwarzian 
derivative, then the closure of the immediate 
basin of any strongly attracting periodic orbit contains either a 
point of the critical line or a point of 
$\Lambda:=\left\{-1,1\right\}\times[0,1]$.
\end{teoBBB}

Before proving this theorem, let us introduce some notation. Write the derivative 
of $F(x,y)=(f(y)(x), K_{sign(x)}(y))$ by
$$ 
DF=\left(
\begin{array}{cc}
f_x & f_y \\
0 & K_y \\
\end{array}
\right). 
$$
For $(x,y)\in\text{Dom}(F)$, let $j_{(x,y)}=(j_1\cdots j_m\cdots)\in\{-,+\}^\N$
such that
\[
F^m(x,y)=\left(f_{j_m}(y_{j_{m-1}\cdots j_1})\circ\cdots\circ
f_{j_2}(y_{j_1})\circ f_{j_1}(y)(x),y_{j_m\cdots j_1}\right), \ \forall m\geq0.
\]
The derivative of $F^m$ at $(x,y)$ is
$$ 
DF^m(x,y)=\left(
\begin{array}{cc}
A^m(x,y) & B^m(x,y) \\
0 & D^m(x,y) \\
\end{array}
\right), 
$$
where
\[
\begin{array}{rcl}
A^m(x,y)&=&\displaystyle\prod_{\ell=0}^{m-1}f_x(F^\ell(x,y)), \\
D^m(x,y)&=&\displaystyle\prod_{\ell=0}^{m-1}K_y(F^\ell(x,y)),\\
B^m(x,y)&=&\displaystyle\sum_{j=0}^{m-1}\left(f_y(F^j(x,y))
\prod_{\ell=0}^{j-1}K_y(F^\ell(x,y))\prod_{\ell=j+1}^{m-1}f_x(F^\ell(x,y))\right).
\end{array}
\]
\newline

\begin{proof}[Proof of Theorem \ref{thesinger222}.]
Let $(p,q)\in\text{Dom}(F)$ be a strongly attracting periodic point of period $m$,
and let $B$ be the closure of its immediate basin. 
By contradiction, we assume that $B$ does not contain points of 
$\Lambda$ neither points of the critical line.
Let $B_0$ be the interior of the connected component of $B$ that contains $(p,q)$. Then
$F^m(B_0)\subset B_0$. Let $[(a,b), q]$ be the connected
component of $B_0\cap I(q)$ containing $(p,q)$, see Figure \ref{B000}.
\begin{figure}[H]\centering
\begin{tikzpicture}
\draw (0.5,0) .. controls (0,-1) and (-1,0) .. (0.5,1);
\draw (0.5,0) .. controls (1,1) and (1.5,1) .. (2,0);
\draw (2,0) .. controls (2.5,-1) and (4.5,-1) .. (4,1);
\draw (4,1) .. controls (3.7,2) and (3,1.5) .. (2,1.5);
\draw (0.5,1) .. controls (1,1.3) and (1,1.5) .. (2,1.5);
\draw[-] (-0.5,0.1) -- (4.5,0.1);
\draw[->,>=triangle 45] (-1,-1) -- (5,-1) node [right] {$x$};
\draw[->,>=triangle 45] (-1,-1) -- (-1,2) node [above] {$y$}; 
\draw[dashed] (4.05,0.1) -- (4.05,-1);
\draw[dashed] (1.945,0.1) -- (1.945,-1);
\fill[black] (3,0.1) circle (0.7 mm) node[above]{$(p,q)$};
\fill[black] (1.945,-1) circle (0.mm) node[below]{$a$};
\fill[black] (4.05,-1) circle (0.mm) node[below]{$b$};
\fill[black] (4.5,0.1) circle (0.mm) node[right]{$I(q)$};
\fill[black] (2.3,1.9) circle (0.mm) node[right]{$B_0$};
\end{tikzpicture}
\caption{Construction of the interval $[(a,b),q]$.}
\label{B000}
\end{figure}
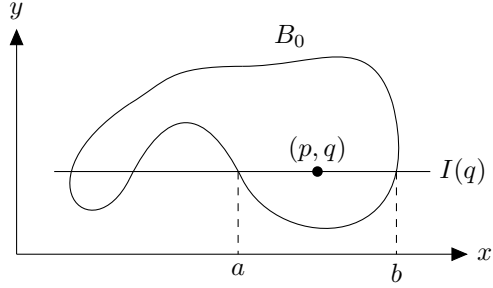 

\noindent
Since $F^m(p,q)=(p,q)$ and $F$ preserves fibers, 
$F^m([(a,b), q])\subset\left([(a,b), q]\right)$.
Let $j=(j_1\cdots j_m)\in\{-,+\}^m$ be the sequence that gives the iterations of $F$ up to order $m$ inside
$[(a,b),q]$, i.e. $F^m(x,q)=\left(g(x), q \right)$ for $x\in(a,b)$, where
$g=f_{j_m}(q_{j_{m-1}\cdots j_1})\circ\cdots\circ f_{j_2}(q_{j_1})\circ f_{j_1}(q)$.
By assumption, we also have $A^m(x,q)\neq 0$.

\medskip
\noindent
{\sc Claim:} $g(\{a,b\})\subset\{a,b\}$. 

\begin{proof}[Proof of the claim.]
Suppose that $g(a)\in(a,b)$ (the case $g(b)\in(a,b)$ is treated analogously).
Since $F^m$ is continuous at $(a,q)$, there is a neighborhood $V\ni(a,q)$
such that $F^m(V)\subset B_0$. This contradicts the assumption that $B_0$ is the
connected component of the immediate basin containing $(p,q)$. The claim is proved.
\end{proof}

We can assume that $A^m(x,q)>0$ for all $x\in(a,b)$ 
and that $g(\alpha)=\alpha$ for $\alpha\in\{a,b\}$
(if this is not the case, take $A^{2m}(x,q)$ instead of $A^m(x,q)$).
Since $(a,q),(b,q)$ are not attracting periodic points, we have
$A^m(\alpha, q)\geq 1$ for $\alpha\in\{a,b\}$. Since $g$ is the composition of
maps with negative Schwarzian derivative, it also has negative Schwarzian 
derivative. By the Minimum Principle, $A^m(x,q)>1$ 
for all $(x,q)\in[(a,b),q]$, which contradicts the equality $F^m([(a,b), q])=[(a,b), q]$.
\end{proof}

We finish this paper with a question. One of the assumptions for Theorem \ref{TheA}
is the non-existence of wandering intervals. For unimodal maps,
negative Schwarzian derivative guarantees the absence of wandering intervals.
This is a famous theorem of Guckenheimer.

\begin{theoremm}[Guckenheimer]
Let $f:I\to I$ be a $C^3$ unimodal map with
negative Schwarzian derivative. If $f''(c)\neq0$ at the unique critical point $c$ of $f$,
then $f$ has no wandering intervals.
\end{theoremm}

\begin{proof}
See \cite[Thm 6.3]{dMvS}.
\end{proof}

We do not know if such implication also holds for toy models, so we ask the following.

\medskip
\noindent
{\bf Question:} Let $F(x,y)=(f(y)(x), K^F_{{\rm sign}(x)}(y))$ be a toy model.
If each unimodal map $f(y)$ has negative Schwarzian 
derivative, can $F$ have wandering intervals?



\bibliography{Kne_Seq}
\bibliographystyle{alpha}

\end{document}